\documentclass{amsart}
\usepackage[enumitem,theorem,hyperref,frak,indicator]{paper_diening}
 \usepackage{geometry}
\newtheorem{mainthm}{Theorem}

\newtheorem*{theorem*}{Theorem}
\usepackage{xcolor}
\usepackage[style=alphabetic,maxbibnames=5,maxalphanames=5, backend=biber,isbn=false,url=false,eprint=false,giveninits=true]{biblatex}  
\definecolor{darkgreen}{rgb}{0,0.6,0}
\definecolor{colorA}{rgb}{0.6,0.0,0.5}

\def\tilde{\widetilde}
\def\kappa{\varkappa}

\numberwithin{equation}{section}

\author{Anna Kh.~Balci}

\author{Mikhail Surnachev }
\address{Anna Kh.~Balci, University Bielefeld, Universit\"atsstrasse 25, 33615
  Bielefeld, Germany.}
\email{akhripun@math.uni-bielefeld.de}

\address{Mikhail Surnachev, Keldysh Institute of Applied Mathematics,  Miusskaya sq. 4, 125047
  Moscow,  Russia.}
\email{peitsche@yandex.ru}

\thanks{Anna Kh.Balci research is funded  by the Deutsche Forschungsgemeinschaft (DFG, German Research Foundation) - SFB 1283/2 2021 - 317210226. Mikhail Surnachev was supported by Moscow Center for Fundamental and Applied Mathematics, Agreement with the Ministry of Science and Higher Education of the Russian Federation, No. 075-15-2019-1623.
}

\keywords{Lavrentiev phenomenon; nonlinear elliptic equations;  double phase potential, generalized Orlicz functions}

\subjclass[2010]{%
35J60, 
46E35, 
35J20, 
35J60. 
}

\title{The Lavrentiev phenomenon in calculus of variations  with differential  forms}
\begin{document}
\maketitle
  
  \begin{abstract}
In this article we study convex non-autonomous variational problems with differential forms and corresponding function spaces.  We introduce a general framework for constructing counterexamples to the Lavrentiev gap, which we apply to several models, including the double phase, borderline case of double phase potential, and variable exponent.  The results for the borderline case of double phase potential provide new insights even for the scalar case, i.e., variational problems with $0$-forms.

  \end{abstract}

\section{Introduction}

In this article we study variational problems and corresponding function spaces associated with the integral functionals of the form
\begin{equation}\label{Funct1}
\mathcal{F}_{\Phi,b}(\omega):= \int\limits_\Omega \Phi(x,|d\omega|)\, dV + \int\limits_\Omega b\wedge d\omega
\end{equation}
where~$\Omega$ is a bounded domain in $\mathbb{R}^N$ (later we will only consider the case of a cube or ball) with~$\Phi:\Omega\times \overline{\mathbb{R}}_{+}\to \overline{\mathbb{R}}_{+}$ is a generalized Orlicz function, $\omega$ a differential $k$-form and $b$ a differential $(N-k-1)$-form, and $dV=dx^1\ldots dx^n$. For~$0$-forms the problem reduces to the classical problem of calculus of variations with~$d\omega$ replaced by $\nabla \omega$. Further we refer to the case of $0$-forms (functions) as the scalar case. 

The classical results on differential forms are collected, for example,  in the  books  by  H.~Cartan~\cite{Car70}, M.~Spivak~\cite{Spivak}, V.I.~Arnold~\cite{Arn89}, H.~Flanders~\cite{Fl89}, R. Abraham, J. E.  Marsden  and T. Ratiu  in~\cite{AbrMarRat}. Iwaniec  and Lutoborski~\cite{IwaLut93}, Iwaniec and Martin~\cite{IwaMar93}, Scott~\cite{Sco95},  Iwaniec, Scott and  Stroffolini~\cite{IwaScoStr99}, Schwarz~\cite{Schwarz}, \textcite{MMS08}, Troyanov~\cite{Tro09} studied Sobolev spaces of differential forms, Gaffney inequalities, and related problems of Hodge theory. More recent results in the framework of Calculus of Variations could be found in books   by Csato, Dacorogna and Kneuss~\cite{CsaGyuDar12} and by   Agarwal,  Ding,  and Nolder~\cite{AgaDinNol09}.  
Recent contributions in the direction of the transportation of closed differential forms were obtained by \textcite{DacGan18, DacGan19};  the optimal  constant in {G}affney inequality was studied by \textcite{CsaDac18}. 

We study calculus of variations for the non-autonomous models with general growth and differential forms. For our knowledge no regularity results are know for such classes. The focus of the present paper is on the conditions, separating the case with the energy gap from the regular case (density of smooth functions) for the integrands with nonstandard growth, in particular, for the variable exponent and double phase models. The study  of the $\rho$-harmonic forms goes back to  \cite{Uhl77} K.~Uhlenbeck, who obtained classical results on the H\"older continuity. These results were extended by Hamburger in~\cite{Ham92}.  Beck and Stroffolini \cite{BecStr13}  considered partial  regularity for  general quasilinear systems for differential forms.  Sil~\cite{Sil17,Sil19,Sil119} studied convexity properties of integral functionals with forms and regularity estimates for inhomogeneous quasilinear systems with forms. Let is mention that the results obtained by Sil are related to the autonomous case~$\Phi(x,\abs{d\omega})=\Phi(\abs{d\omega})$.

In the present paper we study variational problems for the integral functional \eqref{Funct1} with convex integrands~$\Phi(x,t)$ that satisfy general ``nonstandard'' growth conditions of the type 
\begin{align}\label{eq:growth}
 -c_0+c_1\abs{t}^{p_-}\le \Phi(x,t)\le c_2\abs{t}^{p_+}+c_0,
\end{align}
where~$1<p_-\le p_+<\infty$,~$c_0\ge 0$,~$c_1,c_2>0$. The class of ``non-standard'' integrands satisfying~\eqref{eq:growth} includes for example the~$p(x)$-integrand
\begin{equation}\label{pxx}
\Phi(x,t) = t^{p(x)},\quad 1<p_{-} \leq p(x)\leq p_{+} <\infty, \quad x\in \Omega,
\end{equation}
studied for the scalar case in many papers and several books, see~\cite{Zhi86,Zhi95,Zhi11,DieHHR11,CruFio13,KokMesRafSam16}. For the variable exponent model the Hölder regularity of solutions, a Harnack type  inequality for non-negative solutions, and boundary regularity results were obtained by \textcite{Alk97,Alk05} and by \textcite{AlkKra04} under some suitable assumptions on the variable exponent of the~$\log$-Hölder type. Gradient regularity for Hölder exponent was obtained by  \textcite{CosMin99} and for the log-Hölder exponents by \textcite{AceMin01}.

Another classical example of non-standard growth conditions is given by double-phase variational problems which correspond to the functional \eqref{Funct1} with 
\begin{equation}\label{dbl}
\Phi(x,t)=\phi(t) +a(x)\psi(t), \quad a\ge 0,
\end{equation}
where $\varphi$ and $\psi$ are Orlicz functions with different growths rates at infinity. Two notable examples are the ``standard'' double phase model with
\begin{equation}\label{dbl_stand}
\varphi(t) = t^p, \quad \psi(t) = t^q, \quad 1<p<q<\infty,
\end{equation}
and the ``borderline'' double phase model
\begin{equation}\label{borderline}
\varphi(t) = t^p \log^{-\beta}(e+t),\quad \psi(t) = t^p \log^\alpha(e+t).
\end{equation}

Colombo and Mingione in~\cite{ColMin15} obtained  Hölder regularity results for double-phase potential model~$\Phi(x,t)=\frac 1 p t^p+\frac 1 qa(x)t^q$ if $q \leq p(d+\alpha)/d$ and $a \in C^{0,\alpha}(\Omega)$. Moreover, bounded minimizers  are automatically~$W^{1,q}(\Omega)$ if~$a \in C^{0,\alpha}(\Omega)$ and~$q \leq p+\alpha$, see the paper~\cite{BarColMin18} by Baroni, Colombo and Mingione.  As it was shown in~\cite{BalDieSur20} those results in the scalar case are  sharp in terms of the counterexamples on the Lavrentiev gap. 

The special case of the model \eqref{dbl} with $\varphi(t) = t^p$ and $\psi(t) = t^p \log(e+t)$ was studied by Baroni,  Colombo,  Mingione in~\cite{BarColMin15}. In particular they obtained  the ~$C^{0,\gamma}_{\loc}$ regularity result for the minimizers provided that the weight~$a(x)$ is~$\log$-H\"older continuous (with some~$\gamma$) and more strong result (any~$\gamma\in (0,1)$) for the case of vanishing~$\log$-Hölder continuous  weight. Skrypnik and  Voitovych recently proved continuity and Harnack inequality for solutions of a general class of elliptic and parabolic equations with nonstandard growth conditions, see ~\cite{skr20}.  The results on generalized Sobolev-Orlicz spaces are collected in the book by Harjulehto and  H\"ast\"o~\cite{HarHas19} and for anisotropic {M}usielak-{O}rlicz setting in the book by Chlebicka,  Gwiazda, \'{S}wierczewska-Gwiazda and Wr\'{o}blewska-Kami\'{n}ska \cite{ChlGwiSwiWro21}.  In the  general framework of problems with nonstandard growth and nonuniform ellipticity recent results are due to Mingione and R\v{a}dulescu \cite{MinRad21} and to De Filippis and Mingione, see  \cite{DeFilMin20,DeFilMin21a,DeFilMin21}.  Recent contributions  for such energies include new results on density of smooth functions and absence of Lavrentiev gap by \textcite{BulGwiSkr22}, \textcite{koch22global}, and \textcite{BorChlFelBla23}.

An essential feature of the nonautonomous models with nonstandard growth is the presence of the Lavrentiev gap phenomenon. The  energy~$\mathcal{F}_{\Phi,b}$ defines the corresponding generalized partial Sobolev-Orlicz spaces of differential forms~$W^{d,\Phi(\cdot)}(\Omega, \Lambda^k)$ (the natural energy space for $\mathcal{F}_{\Phi,0}$, which consists of forms with $W^{1,1}(\Omega)$ coefficients which fall in the domain of $\mathcal{F}_{\Phi,0}$) described in  Section~\ref{sec:Sobolev}. The Lavrentiev gap in this case is the inequality  
\begin{equation}\label{Lavr}
 \inf  \mathcal{F}_{\Phi,b}(W^{d,\Phi(\cdot)}_c(\Omega,\Lambda^{k})) < \inf  \mathcal{F}_{\Phi,b}(C_0^\infty(\Omega,\Lambda^k)).
\end{equation}
where $W^{d,\Phi(\cdot)}_c(\Omega,\Lambda^{k})$ is the set of $W^{d,\Phi(\cdot)}(\Omega,\Lambda^{k})$ forms compactly supported in $\Omega$. 

A closely related problem is density of smooth functions in the natural energy space of the functional. Denote the closure of smooth forms from $W^{d,\Phi(\cdot)}(\Omega, \Lambda^k)$ in this space by $H^{d,\Phi(\cdot)}(\Omega, \Lambda^k)$. If any function from the domain of $\mathcal{F}_{\Phi,b}$ can be approximated by smooth functions with energy convergence (equivalently, if $H^{d,\Phi(\cdot)}(\Omega, \Lambda^k)=W^{d,\Phi(\cdot)}(\Omega, \Lambda^k)$, which is abbreviated to $H=W$) then the Lavrentiev gap is obviously absent. In the autonomous case, when  the integrand $\Phi=\Phi(t)$ is an Orlicz function independent of $x$, the Lavrentiev phenomenon is absent ($H=W$). 

In the scalar case (for functions = $0$-forms) the study of such models goes back to Zhikov~\cite{Zhi86}, \cite{Zhi95}, who constructed the first examples on Lavrentiev phenomenon for variable exponent model and double phase model in dimension ~$N=2$.   Esposito, Leonetti, Mingione~\cite{EspLeoMin04} generalized this example to any dimension (for the standard double phase model); Fonseca, Mal\'{y}, Mingione \cite{FonMalMin04}  constructed examples of minimizers for the standard double phase model with large (fractal) sets of discontinuity. All these examples required the dimensional  restriction~$p^-<N<p^+$. This restriction was overcome by the authors of the present paper with Diening in~\cite{BalDieSur20} using fractal contact sets for scalar variable exponent, double phase and weighted model. In~\cite{BalSur21} the authors of the present paper studied the Lavrentiev gap property for the borderline double phase model \eqref{dbl} with one saddle point (that is, an example constructed as in \cite{Zhi86}, \cite{Zhi95}, and \cite{EspLeoMin04}) with $p=N$, $\alpha,\beta>0$. 

In this paper we extend the approach of \cite{BalDieSur20} to variational problems with differential forms and refine the construction by using the generalized Cantor sets which have an additional tweaking parameter. This allows  for fine tuning of the singular set, while keeping the formal Hausdorff dimension. We construct examples of the Lavrentiev gap for the $p(x)$-integrand \eqref{pxx} and both ``standard'' double phase \eqref{dbl}, \eqref{dbl_stand} and ``borderline'' double phase \eqref{dbl}, \eqref{borderline} integrands (the last results are new even for the case of scalar functions 0-forms).  For the latter model the fine tuning of the Cantor set is crucial. 

Now we state the main results of this paper. We work with three models: classical double phase potential, borderline double phase potential, and variable exponent. For each of these cases we construct examples for the Lavrentiev gap. However, the construction presented in this paper is not limited to these models. For instance, it can be also used to treat the weighted energy. Let $\Omega$ be a ball in $\mathbb{R}^N$ and $k\in \{0,\ldots,N-2\}$.
\begin{mainthm}\label{theoremA}
Let $p>1$, $\alpha \in [0,1]$, and $q> p + \alpha \max ((k+1)^{-1}, (p-1)(N-k-1)^{-1})$  Then there exists an integrand $\Phi(x,t) = t^p + a(x) t^q$ where nonnegative weight $a=a(x)$ is bounded, $a\in C^\alpha (\overline \Omega)$ if $\alpha\in (0,1)$ and  $a\in \mathrm{Lip}(\overline{\Omega})$ if $\alpha=1$, such that $H^{d,\Phi(\cdot)}(\Omega,\Lambda^{k}) \neq W^{d,\Phi(\cdot)}(\Omega,\Lambda^{k})$ and \eqref{Lavr} holds. 
\end{mainthm}
\begin{mainthm}\label{theoremB}
Let $p_0>1$, $\alpha,\beta \in \mathbb{R}$, $\varkappa\geq 0$ such that $ \alpha+\beta>p_0+\varkappa$. Let $\varphi$ and $\psi$ be two Orlicz functions such that $ \varphi (t) \sim t^{p_0} \ln^{-\beta} t$ and $\psi(t) \sim t^{p_0} \ln ^{\alpha} t$ 
for large $t$. Then there exists an integrand $\Phi(x,t) = \varphi(t) + a(x) \psi(t)$ where $a=a(x)$ is a nonnegative function with the modulus of continuity $C \ln^{-\varkappa}(1/t)$ such that $H^{d,\Phi(\cdot)}(\Omega,\Lambda^{k}) \neq W^{d,\Phi(\cdot)}(\Omega,\Lambda^{k})$ and \eqref{Lavr} holds.
\end{mainthm}
\begin{mainthm}\label{theoremC}
Let $1<p_{-}<p_{+}$. There exists a variable exponent $p:\Omega\to [p_{-},p_{+}]$ with the modulus of continuity $\kappa_0 (\ln t^{-1})^{-1} \ln\ln t^{-1}$, $\varkappa_0=\varkappa_0(p_{-},p_{+},N,k)> 0$, such that for $\Phi(x,t) =t^{p(x)}$ there holds $H^{d,\Phi(\cdot)}(\Omega,\Lambda^{k}) \neq W^{d,\Phi(\cdot)}(\Omega,\Lambda^{k})$ and \eqref{Lavr}.
\end{mainthm}
Theorems \ref{theoremA}, \ref{theoremB}, \ref{theoremC} follow from the theorems~ \ref{theorem:double}, \ref{theorem:bor}, \ref{theorem:var} are proved in Section~\ref{sec:App}. The weight $a=a(x)$ in Theorems \ref{theoremA} and \ref{theoremB} and the exponent $p=p(x)$ in Theorem \ref{theoremC} (as well as the forms providing the examples of non-density and competitors used to show the Lavrentiev phenomenon) are regular outside of a singular set of Cantor type which lies on a proper subspace of $\mathbb{R}^N$. The dimension of this subspace is either $k+1$ or $N-k-1$ depending on the parameters. Compare this to \cite{BalDieSur20} where for the scalar case $k=0$ the singular set was either a Cantor set $\frC$ on a line (``superdimensional'' setup, which was used to construct the examples with variable exponent taking values greater than the space dimension $N$) or a Cantor set $\frC^{N-1}$ on a hyperplane (``subdimensional'' setup, which was used to construct the examples with variable exponent taking values less than the space dimension $N$). For $k$-forms in the variable exponent setting the value of exponent separating these two cases is $N/(k+1)$ --- for exponent taking values greater than $N/(k+1)$ the singular set will be of the form $\frC^{k+1}\times \{0\}^{N-k-1}$ and for exponent taking values less than $N/(k+1)$ the singular set is of the form $\frC^{N-k-1} \times\{0\}^{k+1}$.

Our setting can be called ``semivectorial'', or generalized Uhlenbeck structure, since the integrand is isotropic. In this respect it has  substantially more rigid  structure than the fully vectorial problems (say, of elasticity theory) with quasi-convex integrands. Note that in the ``fully'' vectorial setting the situation is more delicate, and the Lavrentiev phenomenon is possible even for ``standard'' growth conditions in the autonomous (but anisotropic!) case, see Ball, Mizel \cite{BalMiz85} and Foss, Hruza, Mizel~\cite{FosHruMiz03} in the context of non-linear elasticity. 

The  models with Lavrentiev phenomenon are also challenging to study numerically since  the standard numerical schemes fail to converge  to the~$W$-minimiser of the problem. For the scalar case the problem could be solved using non-conforming methods, see \textcite{BalOrtSto22}. The vectorial setting remains open.

\bigskip
\textbf{Structure of the paper.} In Section~\ref{sec:diffforms} we recall some basic definitions related to the theory of differential forms and Sobolev-Orlicz spaces. In Section~\ref{sec:min} we study the existence of minimizers of the functional \eqref{Funct1}. In Section~\ref{sec:framework} we describe the general framework for construction of examples using the fractal Cantor barriers. In Section~\ref{sec:App} we apply this general construction to different problems. We obtain the examples of Lavrentiev gap and non-density of smooth functions  for the classical double phase in Subsection~\ref{subsec:stan_d}, for  the borderline double phase model in  Subsection~\ref{subsec:bor}, and for the  variable exponent model in  Subsection~\ref{subsec:var}. The results for the borderline double phase model are new even in the scalar case.

\section{Differential forms and Sobolev-Orlicz spaces}\label{sec:diffforms} 

Here we recall some basic facts and definitions from the theory of differential forms. In general we follow definitions and notations from \cite{CsaGyuDar12}[Chapters 1.2;2.1;3.1-3.3],  but the Hodge codifferential is the formal adjoint of the exterior derivative $d$ (as in \cite{IwaLut93,GolTro06}). 

\subsection{Exterior algebra.} 

The Grassman algebra of exterior $k$-forms (i.e. skew-symmetric $k$-linear functions) over $\mathbb{R}^N$ is denoted by $\Lambda^k(\mathbb{R}^N)$, or for brevity just by $\Lambda^k$. The exterior product of $f\in \Lambda^k$ and $g\in \Lambda^l$ is defined by
 $$
(f\wedge g) (\xi_1,\ldots,\xi_{k+l}) = \sum \mathrm{sign}\, (i_1,\ldots, i_k, j_1,\ldots,j_l) f(\xi_{i_1}, \ldots, \xi_{i_k}) g(\xi_{j_1}, \ldots, \xi_{j_l})
 $$
 where the summation is over permutations $(i_1,\ldots, i_k, j_1,\ldots,j_l)$ of $(1,2,\ldots,k+l)$ such that $i_1<\ldots < i_k$, $j_1<\ldots<j_l$.  This operation is linear in both arguments, associative, and for $f\in \Lambda^k$ and $g\in \Lambda^l$ there holds $f\wedge g = (-1)^{kl}g\wedge f$.

Let $e_j$ be an orthonormal basis $\{e_j\}_{j=1}^N$ in $\mathbb{R}^N$ and $\{e^j\}_{j=1}^N$ be its dual system in $\Lambda^1$, $e^j(e_l) = \delta_{jl}$. The monomials $e^{i_1}\wedge \ldots \wedge e^{i_k}$, $i_1<i_2<\ldots<i_k$ form a basis in $\Lambda^k$. Denote $f_{i_1\ldots i_k} = f(e_{i_1},\ldots, e_{i_k})$. Then the set of $f_{i_1\ldots i_k}$ with $i_1<i_2<\ldots<i_k$ gives the coordinates of $f$:
$$
f=\sum_{1\leq i_1<\ldots<i_k \leq N} f_{i_1\ldots i_k} e^{i_1}\wedge \ldots \wedge e^{i_k}.
$$

 The scalar product of $f,g\in \Lambda^k$ with coordinates $f_{i_1\ldots i_k}$ and $g_{i_1\ldots i_k}$ is given by
$$
\langle f,g\rangle = \sum_{1\leq i_1<\ldots<i_k \leq N} f_{i_1\ldots i_k} g_{i_1,\ldots i_k}.
$$
The scalar product does not depend on the particular choice of the orthonormal basis $\{e_j\}_{j=1}^N$. We denote $|f| = \langle f,f \rangle ^{1/2}$.  

The Hodge star operator $\ast:\Lambda^k\to \Lambda^{N-k}$ is defined by $f\wedge g = \langle\ast f, g\rangle e^1\wedge \ldots \wedge e^N$ for any $g\in  \Lambda^{N-k} $, or equivalently by $f\wedge \ast g = \langle f,g \rangle e^1 \wedge\ldots \wedge e^N$ for all $f,g\in \Lambda^k$. The Hodge star operator $\ast$ is an isometry between $\Lambda^k$ and $\Lambda^{N-k}$ and for $f\in \Lambda^k$ there holds 
$$
\ast( \ast f)= (-1)^{k(N-k)} f,\quad
\ast^{-1} f = (-1)^{k(N-k)}\ast f.
$$
For any $f\in \Lambda^k$ and shuffle $j_1,\ldots,j_N$  there holds $(\ast f)_{j_{k+1}\ldots j_N} =\mathrm{sign}\, (j_1,\ldots,j_N) f_{j_1\ldots j_k}$.

The interior product (contraction) of $f\in \Lambda^k$ and $g\in \Lambda^l$ defined by
$$
g\lrcorner f = (-1)^{(N-k)(k-l)} \ast (g \wedge (\ast f))
$$
is the adjoint of the wedge product:
$$
\langle g\wedge \alpha, f\rangle = \langle \alpha, g\lrcorner f\rangle\quad \text{for any}\quad \alpha \in \Lambda^{k-l}, \quad f\in \Lambda^k, \quad g\in \Lambda^l.
$$ 
There holds $\ast (g\lrcorner f) = (-1)^{(k-l)l} g\wedge \ast f$ and $\ast (g\wedge f)=(-1)^{kl} g\lrcorner \ast f$.


If $l=0$ then $g\lrcorner f = gf$. If $l>k$ then $g\lrcorner f=0$, if $l=k$ then $g\lrcorner f = f\lrcorner g =(f,g)$. If $l\leq k$ then
$$
(g\lrcorner f)_{j_1\ldots j_{k-l}} = \sum_{1\leq i_1 <\ldots<i_l \leq N} g_{i_1\ldots i_l} f_{i_1\ldots i_l j_1\ldots j_{k-l}}. 
$$

For $w,v \in \Lambda^1$ there holds 
\begin{equation}\label{decomp}
 w \lrcorner (v\wedge f) + v \wedge (w\lrcorner f) = \langle w,v\rangle f.
\end{equation}
For a vector $X$ the operator $\imath_X: \Lambda^k\to \Lambda^{k-1}$ is defined by $(\imath_X f) (\xi_1,\ldots,\xi_{k-1}) = f(X,\xi_1,\ldots,\xi_{k-1})$.  There holds $\imath_v (f\wedge g) = (\imath_v f) \wedge g + (-1)^{\mathrm{deg}\, f} f\wedge (\imath_v g)$, and $\imath_v \imath_w =-\imath_w \imath_v$, $\imath_v\imath_v =0$. For a vector $v\in \mathbb{R}^N$ and the 1-form $v^\flat\in \Lambda^1$ with the same coordinates there holds $v^\flat\lrcorner f = \imath_v f$. 


\subsection{Differential forms.}

A differential form is a mapping from $\Omega\subset \mathbb{R}^N$ to $\Lambda^k$. Further $\Omega$ will be a bounded contractible domain with sufficiently regular boundary. Using the canonical basis $dx^{i_1}\wedge \ldots \wedge dx^{i_k}$ a $k$-form can be represented as
\begin{equation}\label{fform1}
f=\sum_{1\leq i_1<\ldots<i_k \leq N} f_{i_1\ldots i_k} dx^{i_1}\wedge \ldots \wedge dx^{i_k}.
\end{equation}
Then $|f|^2(x):= \sum_{1\leq i_1<\ldots<i_k \leq N} |f_{i_1\ldots i_k}|^2(x)$.

For two differential forms $f$ and $g$ of order $k$ their scalar product in the sense of $L^2(\Omega, \Lambda^k)$ is
 $$
 (f,g) = \int\limits_\Omega f \wedge \ast g = \int\limits_\Omega \langle f,g\rangle \, dV,\quad dV=dx^1\wedge \ldots \wedge dx^N.
 $$
The operation of exterior differentiation $d$ is a unique mapping from $k$-forms to $(k+1)$-forms such that $df$ coincides with the differential of $f$ for $0$-forms (functions), $d\circ d=0$, $d(\alpha\wedge \beta) = d\alpha\wedge \beta + (-1)^{k} \alpha \wedge d\beta$ for any $\alpha\in C^1(\Omega,\Lambda^k)$ and $\beta\in C^1(\Omega,\Lambda^l)$. For a $k$-form $f$, 
$$
df (\xi_1,\ldots, \xi_{k+1}) = \sum_{j=1}^{k+1} (-1)^{j-1} [f'(x) \xi_j] (\xi_1,\ldots, \widehat{\xi_j}, \ldots, \xi_{k+1}), \quad \xi_1,\ldots,\xi_{k+1}\in\mathbb{R}^N.
$$
The interior derivative (Hodge codifferential) of a $k$-form $f$ is 
$$
\delta f = (-1)^{N(k-1)+1} \ast d \ast f = (-1)^{k} \ast^{-1} d \ast f.
$$
There holds $d^2=0$, $\delta^2=0$. On $k$-forms
\begin{equation}\label{hodge_diff}
\ast \delta = (-1)^k d \ast \quad\text{and} \quad \ast d = (-1)^{k+1}\delta \ast.
\end{equation} 
For a $k$-form $f$ and $l$-form $g$ there holds 
$$
d (f\wedge g) = df\wedge g + (-1)^k f\wedge dg, \quad \delta (f \lrcorner g) = (-1)^{k+l+1} df \lrcorner g - f \lrcorner \delta g.
$$

Formally one can write $df = \nabla \wedge f$, $\delta f = - \nabla \lrcorner f$, and in coordinates, for the form \eqref{fform1}, using the Einstein convention of summation over repeated indices we have 
\begin{equation}\label{diff_comp}
(df)_{i_1\ldots i_{k+1}} = (-1)^{l-1} \partial_{x_{i_l}} f_{i_1\ldots \hat\imath_l\ldots i_{k+1}}, \quad (\delta f)_{i_1\ldots i_{k-1}} = -\partial_{x_j} f_{ji_1\ldots i_{k-1}}.
\end{equation}

Let $\nu=(\nu_1,\ldots,\nu_N)$ be the unit outer normal to $\Omega$ and $\nu^\flat = \nu_1 dx^1 +\ldots+\nu_N dx^N$. For a differential $k$-form $f$ the standard Gauss theorem reads as (see \eqref{diff_comp}) 
$$
\int\limits_\Omega (df)_{i_1\ldots i_{k+1}} d V = \int\limits_{\partial \Omega} (\nu^\flat \wedge f)_{i_1\ldots i_{k+1}} d \sigma, \quad
\int\limits_\Omega (\delta f)_{i_1\ldots i_{k-1}}dV = -\int\limits_{\partial \Omega} (\nu^\flat \lrcorner f)_{i_1\ldots i_{k-1}} d\sigma 
$$
for each $1\leq i_1 <\ldots <i_k \leq N$, where $dV=dx^1\ldots dx^N$ is the standard volume form and $d\sigma$ is the surface area element. The same reasoning also gives the integration-by-parts formula
\begin{equation}\label{green00}
\int\limits_\Omega\langle df,g\rangle \, dV - \int\limits_\Omega \langle f,\delta g\rangle\, dV
=\int\limits_{\partial \Omega}\skp{\nu^\flat\wedge f}{g}\, d\sigma=\int\limits_{\partial \Omega}\skp{f}{\nu^\flat \lrcorner g}\, d\sigma.
\end{equation}

In the sense of forms, the surface element $d\sigma$ is connected to the volume form $dV$ by $d\sigma = \imath_\nu dV$. The orientation of $\partial \Omega$ is chosen such that the integral of $d\sigma$ over any ``substantial'' boundary part is positive.

The operators $d$ and $\delta$ are adjoint on compactly supported forms. By direct computation (use \eqref{decomp} with $v=\nabla$, $w=-\nabla$), $d\delta + \delta d = -\triangle$, where the Laplace operator is applied componentwise.



A form $f$ satisfying $df=0$ is closed. A form $f$ satisfying $\delta f =0$ is coclosed. If $f=dg$ then $f$ is exact, and if $f=\delta g$ then $f$ is coexact. If both $df=0$ and $\delta f = 0$ the form is called harmonic (or harmonic field).

The pullback of the form $f$ under the mapping $\varphi$ is defined by $\varphi^*f$, 
$$
(\varphi^* f)(x;\xi_1,\ldots,\xi_k)=f(\varphi(x);\varphi'(x)\xi_1,\ldots, \varphi'(x)\xi_k).
$$
This operation satisfies $\varphi^* (\alpha\wedge \beta) = \varphi^* \alpha \wedge \varphi^* \beta$ and $\varphi^* d = d\varphi^*$.


\subsection{Tangential and normal part of a form.} Let $\nu=(\nu_1,\ldots,\nu_N)$ be the unit outer normal to $\partial \Omega$ and for a differential form $\omega$ define its tangential part
$$
t \omega  (\xi_1,\ldots,\xi_k)= \omega (\xi_1 - (\nu,\xi_1)\nu , \ldots, \xi_k - (\nu,\xi_k) \nu)
$$
and its normal part $n\omega = \omega - t\omega$. Define the $1$-form $\nu^\flat = \nu_1 dx^1+\ldots+\nu_N dx^N$, then 
\begin{gather*}
tf =\nu^\flat \lrcorner (\nu^\flat \wedge f), \quad nf = \nu^\flat\wedge (\nu^\flat \lrcorner f),\quad f=tf+nf,\\
\nu^\flat \wedge t f = \nu^\flat\wedge f,\quad \nu^\flat \lrcorner\, tf = 0,\quad tf=0 \Leftrightarrow \nu^\flat \wedge f =0,\\
\nu^\flat \wedge nf =0, \quad \nu^\flat \lrcorner\, nf = \nu^\flat \lrcorner f, \quad nf =0 \Leftrightarrow \nu^\flat \lrcorner f=0.
\end{gather*}
That is, setting $tf$ is equivalent to setting $\nu^\flat \wedge f$ and setting $nf$ is equivalent to setting $\nu^\flat\lrcorner f$.  

While integrating over $\partial\Omega$, the tangential part of a form coincides with its pullback under the inclusion $\jmath: \partial \Omega \to \overline \Omega$, that is $t\omega = \jmath^* \omega$, and the normal part of the form vanishes.

The decomposition of a form into the tangential and normal parts can be also done using coordinates in a ``collar'' neighbourhood of $\partial \Omega$, by choosing (locally) an ``admissible'' coordinate system (map) $\varphi:U\to V$, $U,V\subset \mathbb{R}^N$, such that $\partial \Omega \cap V \subset \{\varphi(y',0)\,:\, (y',0)\subset U\}$ and $(\varphi_{y_i}(y',0), \varphi_{y_N} (y',0))=\delta_{iN}$, $1\leq i \leq N$. In this coordinate system for
$$
\omega = \sum_{1\leq i_1<\ldots<i_k\leq N} \omega_{i_1\ldots i_k} dy^{i_1}\wedge \ldots \wedge dy^{i_k}
$$
we have $\omega = t\omega + n\omega$, where
\begin{align*}
t\omega &= \sum_{1\leq i_1<\ldots<i_k<N} \omega_{i_1\ldots i_k} dy^{i_1}\wedge \ldots \wedge dy^{i_k} ,\\
n\omega &=  \sum_{1\leq i_1<\ldots<i_{k-1}<N} \omega_{i_1\ldots i_k} dy^{i_1}\wedge \ldots \wedge dy^{i_{k-1}}\wedge dy^N.
\end{align*}
If $\omega$ is a function ($0$-form) we set $t\omega = \omega$ and $n\omega=0$. If $\omega$ is an $N$-form we set $t\omega=0$, $n\omega=\omega$.  The decomposition $\omega = t\omega + n\omega$ is invariant on $\partial \Omega$ and we have
\begin{gather*}
t\ast = \ast n,\quad n\ast = \ast t, \quad td=dt,\quad n\delta = \delta n.
\end{gather*}
In particular, $t\omega=0$ on $\partial \Omega$ implies $td\omega=0$ on $\partial \Omega$ and $n\omega =0$ on $\partial \Omega$ implies $n\delta \omega=0$ on $\partial \Omega$. 

In terms of the Stokes theorem, integration-by-parts formula \eqref{green00} reads as follows: by \eqref{hodge_diff} for a $k$-form $f$ and a $(k+1)$-form $g$ there holds $d(f\wedge \ast g) = df \wedge \ast g - f \wedge \ast \delta g$, therefore 
\begin{equation}\label{green0}
(d f,g) -  (f,\delta g) = \int \limits_\Omega \bigl( df \wedge \ast g - f \wedge \ast \delta g \bigr) = \int\limits_\Omega d (f \wedge \ast g) =
\int\limits_{\partial\Omega} f \wedge \ast g  = \int\limits_{\partial\Omega} tf\wedge \ast ng.
\end{equation}

\subsection{Orlicz functions setup} \label{sect:Orlicz}

We say that~$\oldphi\,:\, [0,\infty) \to [0,\infty]$ is an Orlicz
function if $\oldphi$ is convex, left-continuous, $\oldphi(0)=0$,
$\lim\limits_{t \to 0} t^{-1}\oldphi(t)=0$ and
$\lim\limits_{t \to \infty} t^{-1} \oldphi(t)=\infty$.  The conjugate Orlicz
function~$\oldphi^*$ is defined by
\begin{align*}
  \oldphi^*(s) &:= \sup_{t \geq 0} \big( st - \oldphi(t)\big).
\end{align*}
In particular, $st \leq \oldphi(t) + \oldphi^*(s)$.

In the following we assume that
$\Phi\,:\, \Omega \times [0,\infty) \to [0,\infty]$ is a generalized
Orlicz function, i.e. $\Phi(x, \cdot)$ is an Orlicz function for
every~$x \in \Omega$ and $\Phi(\cdot,t)$ is measurable for
every~$t\geq 0$. We define the conjugate function~$\Phi^*$ point-wise,
i.e.  $\Phi^*(x,\cdot) := (\Phi(x,\cdot))^*$.

We assume that $\Phi$ satisfies the ``nonstandard'' growth condition
\begin{align}\label{eq:growth1}
 -c_0+c_1\abs{t}^{p_-}\le \Phi(x,t)\le c_2\abs{t}^{p_+}+c_0,
\end{align}
where~$1<p_-\le p_+<\infty$,~$c_0\ge 0$,~$c_1,c_2>0$, and the following properties:
\begin{enumerate}
\item $\Phi$ satisfies the $\Delta_2$-condition,
  i.e. there exists~$c \geq 2$ such that for
  all~$x \in \Omega$ and all~$t\geq 0$
  \begin{align}
    \label{eq:phi-Delta2}
    \Phi(x,2t) &\leq c\, \Phi(x,t). 
  \end{align}
\item $\Phi$ satisfies the~$\nabla_2$-condition,
  i.e. $\Phi^*$ satisfies the~$\Delta_2$-condition. As a consequence,
  there exist~$s>1$ and~$c> 0$ such that for all $x\in \Omega$,
  $t\geq 0$ and $\gamma \in [0,1]$ there holds
  \begin{align}
    \label{eq:phi-Nabla2}
    \Phi(x,\gamma t) \leq c\,\gamma^s \,\Phi(x,t).
  \end{align}
\item $\Phi$ and~$\Phi^*$ are proper, i.e. for every~$t\geq 0$ there holds
$$
\int\limits_\Omega \Phi(x,t)\,dV< \infty \quad \text{and} \quad \int\limits_\Omega \Phi^*(x,t)\,dV < \infty.
$$

\end{enumerate}

\subsection{Sobolev-Orlicz spaces of differential forms}\label{sec:Sobolev}

Let~$\Omega\subset \setR^N$ be a bounded domain in $\mathbb{R}^N$. In our applications this will always be a ball or a cube. 

Different functional spaces like Lebesgue spaces~$L^p(\Omega)$ and Lebesgue-Orlicz spaces $L^{\varphi(\cdot)}(\Omega)$, Sobolev spaces~$W^{1,p}(\Omega)$ and Sobolev-Orlicz spaces $W^{1,\Phi(\cdot)}(\Omega)$, spaces of $k$ times continuosly differentiable functions~$C^{k}(\Omega)$ are defined in the usual way (see, for example, \cite{HarHas19}). 

The Lebesgue-Orlicz space $L^{\Phi(\cdot)}(\Omega)$ is the set of all measurable functions in $\Omega$ with finite Luxemburg norm
$$
\|f\|_{L^{\Phi(\cdot)}(\Omega)} = \inf \biggl\{ \lambda>0\,:\, \int\limits_\Omega \Phi(x,|f|\lambda^{-1})\,dV \leq 1 \biggr\}.
$$

The Sobolev-Orlicz space $W^{1,\Phi(\cdot)}(\Omega)$ is the set of functions $f\in W^{1,1}(\Omega)$ such that $|\nabla f| \in L^{\Phi(\cdot)}(\Omega)$, endowed with the norm $\|f\|_{W^{1,\Phi(\cdot)}(\Omega)} = \|f\|_{L^1(\Omega)} + \|\nabla f\|_{L^{\Phi(\cdot)}(\Omega)}$.




For a generalized Orlicz function $\Phi(x,t)$ we define the Lebesgue-Orlicz space $L^{\Phi(\cdot)}(\Omega, \Lambda^k)$ as the space of measurable differential $k$-form such that $|f| \in L^{\Phi(\cdot)}(\Omega)$. The norm in this space is the norm of $|f|$ in  $L^{\Phi(\cdot)}(\Omega)$. For constant $\Phi\equiv p \geq 1$ we get the standard Lebesgue space $L^p(\Omega,\Lambda^k)$.



Let~$r\in \mathbb{N}\cup \{0\}$. For $E = \Omega$ or $E = \overline \Omega$  the space~$C^r(E,\Lambda^k)$ is the space of all differential~$k$-forms for which all partial derivatives~$D^\alpha f^I$ up to the order~$r$ are continuous in $E$. By~$C_0^\infty(\Omega,\Lambda^k)$ we denote the space of smooth $k$-forms with compact support in $\Omega$. 

\begin{definition}[Full Sobolev-Orlicz Space]
We say that a $k$-form~$f\in W^{1,\Phi(\cdot)}(\Omega,\Lambda^k)$ if~$f_{i_1\cdots i_k}\in W^{1,\Phi(\cdot)}(\Omega)$ for every~$1\le i_1<\cdots<i_k\le N$.  The norm is defined componentwise:
$$
\|f\|_{W^{1,\Phi(\cdot)}(\Omega, \Lambda^k)} = \sum_{1\leq i_1<\ldots<i_k\leq N} \|f_{i_1\ldots i_k}\|_{W^{1,\Phi}(\Omega)}.
$$
\end{definition}


If $\Omega$ is of class $C^2$ then for $f\in W^{1,1}(\Omega,\Lambda^k)$ the boundary trace of $f$ exists, belongs to $L^1(\partial \Omega)$, and the Stokes theorem holds (see~\cite{Schwarz}, \cite[Theorem 6.4]{Ren13}):
$$
\int\limits_{\partial \Omega}  f=\int\limits_{\Omega} df.
$$

We say that $u \in L^1_{loc} (\Omega, \Lambda^k)$ has a {\it weak differential}\, $du \in L^1_{loc} (\Omega, \Lambda^{k+1})$ if for any $\xi \in C_0^\infty (\Omega, \Lambda^{k+1} )$ there holds $
(u,\delta \xi)= (du,\xi)$ in $L^2(\Omega, \Lambda^k)$, or equivalently
$$
\int\limits_\Omega u \wedge d\xi = (-1)^{k+1} \int\limits_\Omega du \wedge \xi
$$
for any $\xi \in C_0^\infty (\Omega, \Lambda^{N-k-1})$. 

We say that $u \in L^1_{loc} (\Omega, \Lambda^k)$ has a {\it weak codifferential}\, $\delta u \in L^1_{loc} (\Omega, \Lambda^{k-1})$ if for any $\xi \in C_0^\infty (\Omega, \Lambda^{k-1})$ there holds $(u,d\xi) = (\delta u,\xi)$ in $L^2(\Omega, \Lambda^k)$, or equivalently 
$$
\int\limits_\Omega u\wedge \delta \xi = (-1)^{k} \int\limits_\Omega \delta u \wedge \xi,
$$
for any  $\xi \in C_0^\infty (\Omega, \Lambda^{N-k+1})$.

Both weak differential and codifferential are unique.

\begin{definition}[Partial Sobolev-Orlicz Space]\label{def:Sobolev}
 For~$0\le  k\le N-1$  we define the partial Sobolev-Orlicz space $W^{d,\Phi(\cdot)}(\Omega, \Lambda^k)$ as the set of forms $\omega \in L^1(\Omega, \Lambda^k)$ with weak differential $d\omega \in L^{\Phi(\cdot)}(\Omega,\Lambda^k)$, endowed with the norm
\begin{align*}
\norm{\omega}_{W^{d,\Phi(\cdot)}(\Omega,\Lambda^k)}:=\norm{\omega}_{L^1(\Omega,\Lambda^k)}+\norm{d\omega}_{L^{\Phi(\cdot)}(\Omega,\Lambda^{k+1})}.
\end{align*} 
The space $H^{d,\Phi(\cdot)}(\Omega, \Lambda^k)$ is the closure of smooth forms from $ W^{d,\Phi(\cdot)}(\Omega, \Lambda^k)$ in this space.   

For~$1\le k\le N$  we define  the space $W^{\delta,\Phi(\cdot)}(\Omega, \Lambda^k)$ as the set of forms $\omega \in L^1(\Omega, \Lambda^k)$ with weak codifferential $\delta \omega \in L^{\Phi(\cdot)}(\Omega,\Lambda^k)$ endowed with the norm
\begin{align*}
 \norm{\omega}_{W^{\delta,\Phi(\cdot)}(\Omega,\Lambda^k)}:=\norm{\omega}_{L^1(\Omega,\Lambda^k)}+\norm{\delta\omega}_{L^{\Phi(\cdot)}(\Omega,\Lambda^{k-1})}.
\end{align*} 
The space  $H^{\delta,\Phi(\cdot)}(\Omega, \Lambda^k)$ is the closure of smooth forms from $ W^{\delta,\Phi(\cdot)}(\Omega, \Lambda^k)$ in this space. 
\end{definition}

If $\Omega$ is a bounded $C^2$ domain (or a polyhedral domain), then the following Green's formulas hold \cite[Theorem 3.28]{CsaGyuDar12}. Let $0\le k\le N-1$ and let~$p>1$. If ~$f\in W^{d,p}(\Omega,\Lambda^k)$,~$g\in W^{1,p'}(\Omega,\Lambda^{k+1})$, then
\begin{align*}
 \int\limits_\Omega \skp{df}{g}\, dV-\int\limits_\Omega \skp{\delta g}{f}\, dV=\int\limits_{\partial \Omega}\skp{\nu\wedge f}{g}\, d\sigma. 
\end{align*}
If ~$f\in W^{1,p}(\Omega,\Lambda^k)$,~$g\in W^{\delta,p'}(\Omega,\Lambda^{k+1})$, then
\begin{align*}
 \int\limits_\Omega \skp{df}{g}\, dV-\int\limits_\Omega \skp{\delta g}{f}\,dV=\int\limits_{\partial \Omega}\skp{f}{\nu \lrcorner g}\,d\sigma. 
 \end{align*}

The boundary traces $\nu\wedge f$ and $\nu \lrcorner g$ in these formulas are given by bounded linear mappings from $W^{d,p}(\Omega,\Lambda^k)$ to $W^{-1/p,p}(\partial \Omega, \Lambda^{k+1})$ and from $W^{\delta,p'}(\Omega,\Lambda^{k+1})$ to $W^{-1/p',p'}(\partial \Omega, \Lambda^{k})$, respectively. These mappings are generated by these very integration-by-parts formulas. If $f$ belongs to the full Sobolev space $W^{1,p}(\Omega,\Lambda^k)$, then both tangential and normal components of its boundary trace $tf$ and $nf$ are from $W^{1-1/p,p}(\partial\Omega,\Lambda^k)$.  

Let $W^{d,\Phi(\cdot)}_{c}(\Omega, \Lambda^k)$ be the set of forms from $W^{d,\Phi(\cdot)}(\Omega, \Lambda^k)$ with compact support in $\Omega$.

\begin{definition}[Spaces with zero tangential component]\label{def:SobolevT}

 For~$0\le k\le N-1$ we define the space $ W^{d,\Phi(\cdot)}_{T}(\Omega, \Lambda^k)$ as the set of $\omega \in W^{d,\Phi(\cdot)}(\Omega, \Lambda^k)$ such that $(d\omega,\beta)_{L^2(\Omega,\Lambda^{k+1})}=(\omega,\delta\beta)_{L^2(\Omega,\Lambda^k)}$  for all $\beta\in C^1(\overline{\Omega}, \Lambda^{k+1})$, endowed with the norm of $W^{d,\Phi(\cdot)}(\Omega, \Lambda^k)$.

 
 The space $\tilde W^{d,\Phi(\cdot)}_{T}(\Omega, \Lambda^k)$ is the closure of $W^{d,\Phi(\cdot)}_{c}(\Omega, \Lambda^k)$ in $W^{d,\Phi(\cdot)}(\Omega, \Lambda^k)$.




The space $H^{d,\Phi(\cdot)}_T(\Omega, \Lambda^k)$ is the closure of $C_0^\infty(\Omega,\Lambda^k)$ in $W^{d,\Phi(\cdot)}(\Omega, \Lambda^k)$.


\end{definition}

Clearly, $H^{d,\Phi(\cdot)}_T(\Omega, \Lambda^k) \subset \tilde W^{d,\Phi(\cdot)}_{T}(\Omega, \Lambda^k) \subset W^{d,\Phi(\cdot)}_{T}(\Omega, \Lambda^k)$. A smooth $k$-form $\omega$ belongs to $H^{d,\Phi(\cdot)}_T(\Omega, \Lambda^k)$ if and only if its tangential component $t\omega$ is zero on $\partial \Omega$.

Let $W^{\delta,\Phi(\cdot)}_{c}(\Omega, \Lambda^k)$ be the set of forms from $W^{\delta,\Phi(\cdot)}(\Omega, \Lambda^k)$ with compact support in $\Omega$.

 \begin{definition}[Spaces with zero normal component]\label{def:SobolevN}

 For~$1\le k\le N$ we define the space\\ $W^{\delta,\Phi(\cdot)}_{ N}(\Omega, \Lambda^k)$ as the set of $\omega \in W^{\delta,\Phi(\cdot)}(\Omega, \Lambda^k)$ such that $ (\delta\omega,\beta)_{L^2(\Omega,\Lambda^{k-1})}=(\omega,d\beta)_{L^2(\Omega,\Lambda^{k})}$ for all $ \beta\in C^1(\overline{\Omega}, \Lambda^{k-1})$,  endowed with the norm of $W^{\delta,\Phi(\cdot)}(\Omega, \Lambda^k)$. 
 
The space $\tilde W^{\delta,\Phi(\cdot)}_{N}(\Omega, \Lambda^k)$ is the closure of $W^{\delta,\Phi(\cdot)}_{c}(\Omega, \Lambda^k)$ in $W^{\delta,\Phi(\cdot)}(\Omega, \Lambda^k)$.


 
%
 The space $H^{\delta,\Phi(\cdot)}_N(\Omega, \Lambda^k)$ is the closure of $C_0^\infty(\Omega,\Lambda^k)$ in $W^{\delta,\Phi(\cdot)}(\Omega, \Lambda^k)$.
%
%
%

 \end{definition}

Clearly, $H^{\delta,\Phi(\cdot)}_N(\Omega, \Lambda^k) \subset \tilde W^{\delta,\Phi(\cdot)}_{N}(\Omega, \Lambda^k) \subset W^{\delta,\Phi(\cdot)}_{ N}(\Omega, \Lambda^k)$. A smooth $k$-form $\omega$ belongs to $H^{\delta,\Phi(\cdot)}_N(\Omega, \Lambda^k)$ if and only if its normal component component $n\omega$ is zero on $\partial \Omega$.

The following proposition is straightforward.

\begin{proposition}
All the spaces introduced in Definitions~\ref{def:Sobolev}, \ref{def:SobolevT}, \ref{def:SobolevN} are Banach spaces. 
\end{proposition}

For $\Phi(\cdot)\equiv p \in [1,\infty]$ we get the classical partial Sobolev spaces $W^{d,p}_T(\Omega, \Lambda^k)$ and $W^{\delta,p}_{N}(\Omega, \Lambda^k)$ with vanishing tangential (correspondingly normal) component on the boundary. The spaces $W^{d,p}_T(\Omega, \Lambda^k)$ and $W^{\delta,p}_{N}(\Omega, \Lambda^k)$ coincide with the closures of $C_0^\infty(\Omega, \Lambda^k)$ in $W^{d,p}(\Omega, \Lambda^k)$ and $W^{\delta,p}(\Omega, \Lambda^k)$, respectively (see \cite{IwaLut93}). In this case there is no difference between between $H$ and $W$ spaces.

If $\overline\Omega\Subset \Omega'$, a form from $W^{d,\Phi(\cdot)}(\Omega, \Lambda^k)$ or $W^{\delta,\Phi(\cdot)}(\Omega, \Lambda^k)$ belongs to $W^{d,\Phi(\cdot)}_T(\Omega, \Lambda^k)$ or $W^{\delta,\Phi(\cdot)}_{ N}(\Omega, \Lambda^k)$, respectively, iff its extension by zero to $\Omega'\setminus \Omega$ produces an element from $W^{d,\Phi(\cdot)}(\Omega', \Lambda^k)$ or $W^{\delta,\Phi(\cdot)}(\Omega', \Lambda^k)$, respectively.

\section{Minimization problem for non-autonomous functionals with differential forms}\label{sec:min}

\subsection{Gauge fixing}
Recall~(for instance, \cite[Theorem 6.5]{CsaGyuDar12}) the following facts regarding the harmonic forms with vanishing tangential or normal components at the boundary. Let $\mathcal{H}_T(\Omega,\Lambda^k)$ be the set of harmonic forms from $W^{1,2}_T(\Omega,\Lambda^k)$ and $\mathcal{H}_N(\Omega,\Lambda^k)$ be the set of harmonic forms from $W^{1,2}_N(\Omega,\Lambda^k)$. The spaces $\mathcal{H}_T(\Omega,\Lambda^k)$ and $\mathcal{H}_N(\Omega,\Lambda^k)$ are finite dimensional, closed in $L^2(\Omega,\Lambda^k)$, for contractible domains $\mathcal{H}_T(\Omega,\Lambda^k) = \{0\}$ for $0\leq k \leq N-1$ and $\mathcal{H}_N(\Omega,\Lambda^k) = \{0\}$ for $1\leq k \leq N$. The space $\mathcal{H}_T(\Omega,\Lambda^N)$ is the span of $dx^1\wedge\ldots \wedge dx^N$ and the space $\mathcal{H}_N(\Omega,\Lambda^0)$ is the span of $1$. 

We  need the following result on the solvability of the Cauchy-Riemann type systems for differential forms. This result is a particular case of theorems \cite[Theorem 7.2]{CsaGyuDar12} for~$p\ge 2$ and~\cite[Theorems 2.43]{SilT16} for any~$p>1$, and triviality of the set of harmonic forms with zero tangential component at the boundary. 



\begin{corollary}\label{Forms:existence1}
Let $\Omega$ be a bounded contractible $C^3$ domain in $\mathbb{R}^N$, $0\leq k\leq N-1$, $p>1$,  $\omega_0\in W^{1,p}(\Omega,\Lambda^k)$, and $\beta\in \omega_0+ W^{d,p}_T(\Omega,\Lambda^{k})$. The problem 
\begin{gather*}
d\omega = d\beta, \quad \delta \omega =0  \quad \text{in}\quad \Omega,\\
\nu^\flat \wedge \omega = \nu^\flat \wedge \omega_0 \quad \text{on}\quad \partial \Omega
\end{gather*}
has a unique solution $\omega \in W^{1,p}(\Omega, \Lambda^k)$ with
$$
\|\omega\|_{W^{1,p}(\Omega, \Lambda^k)} \leq C \left(\|\omega_0\|_{W^{1,p}(\Omega, \Lambda^k)} +\|d\beta\|_{L^p(\Omega,\Lambda^{k+1})}\right).
$$
with $C=C(N,p,\Omega)$. 
\end{corollary}

\subsection{Existence of minimizers}

In this section $0\le k\le N-1$, ~$\Omega$ is a bounded contractible $C^3$ domain in $\setR^N$,  $\Phi:\Omega \times[0,+\infty)\to [0,+\infty)$ is a generalized Orlicz function satisfying \eqref{eq:growth1}, $b\in L^{\Phi^*(\cdot)}(\Omega,\Lambda^{N-k-1})$. We study the existence of solutions to the following two variational problems.

\noindent (W-minimization). Let $\omega_0\in W^{1,\Phi(\cdot)}(\Omega, \Lambda^{k})$ and minimize $\mathcal{F}_{\Phi,b}$ over the set $\omega_0+ W^{1,\Phi(\cdot)}_T(\Omega,\Lambda^k)$:
\begin{equation}\label{VarProbl}
 \mathcal F_{\Phi,b}(\omega)=\int\limits_\Omega \Phi(x,|d\omega|)\, dV + \int\limits_\Omega b\wedge d\omega\ \rightarrow \min, \quad \omega\in \omega_0+ W^{1,\Phi(\cdot)}_T(\Omega,\Lambda^k).
\end{equation}

\noindent (H-minimization) Let $\omega_0\in H^{1,\Phi(\cdot)}(\Omega, \Lambda^{k})$ and minimize $\mathcal{F}_{\Phi,b}$ over the set $\omega_0+ H^{1,\Phi(\cdot)}_T(\Omega,\Lambda^k)$:
\begin{equation}\label{VarProbl1}
 \mathcal F_{\Phi,b}(\omega) \rightarrow \min, \quad \omega\in \omega_0+ H^{1,\Phi(\cdot)}_T(\Omega,\Lambda^k).
\end{equation}


\begin{theorem}\label{th:existW}
The variational problem \eqref{VarProbl} has a minimizer $\omega \in \omega_0+W^{d,\Phi(\cdot)}_{T}(\Omega,\Lambda^{k})$ with $\delta \omega=0$.
\end{theorem}


\begin{proof}
    Let~$\omega_s$ be a minimizing sequence, clearly  \begin{align*}
    \norm{d\omega_s}_{L^{\Phi(\cdot)}(\Omega, \Lambda^{k+1})}\le c.
    \end{align*}      
Due to the coercitivity condition~\eqref{eq:growth} we have~
    \begin{align*}
    \norm{d\omega_s}_{L^{p_{-}}(\Omega, \Lambda^{k+1})}\le c.
    \end{align*}
By Corollary~\ref{Forms:existence1} there exists~$\alpha_s\in \omega_0+ W^{1,p_{-}}_{T}(\Omega, \Lambda^{k})$ satisfying $d\alpha_s=d\omega_s$ and $\delta \alpha_s=0$ in $\Omega$
such that
\begin{align*}
    \norm{\alpha_s}_{W^{1,p_{-}}(\Omega,\Lambda^{k})}\leq c.
\end{align*}
Clearly~$\alpha_s\in W^{d,\Phi(\cdot)}(\Omega,\Lambda^{k})$  and
 \begin{align*}
    \norm{\alpha_s}_{W^{d,\Phi(\cdot)}(\Omega,\Lambda^{k})} \le c.
    \end{align*}
The sequence $\alpha_s$ is bounded in the space $W^{d,\Phi(\cdot)}(\Omega, \Lambda^{k}) \cap W^{1,p_{-}}(\Omega,\Lambda^{k})$ endowed with the norm which is the sum of norms in $W^{d,\Phi(\cdot)}(\Omega, \Lambda^{k})$ and $W^{1,p_{-}}(\Omega,\Lambda^{k})$. Its dual space is separable, therefore there exists
$$
\alpha\in\omega_0+ W^{d,\Phi(\cdot)}_{ T}(\Omega, \Lambda^{k}) \cap W^{1,p_{-}}(\Omega,\Lambda^{k})
$$
such that $\delta\alpha=0$ and up to the subsequence,
\begin{gather*}
    \alpha_s\to \alpha\quad \text{in }\quad L^{p_{-}}(\Omega,\Lambda^{k}),\\
d\alpha_s \rightharpoonup d\alpha    \quad \text{in }\quad L^{\Phi(\cdot)}(\Omega,\Lambda^{k+1}). 
\end{gather*}
    Due to the lower-semicontinuity of $\mathcal{F}$, which follows from the convexity of~$\Phi(x,\cdot)$ and Mazur's lemma,  we have
        \begin{align*}
            \liminf_{s\to \infty} \int\limits_\Omega \Phi(x,|d\alpha_s|)\, dV \ge \int\limits_\Omega \Phi(x,|d\alpha|)\, dV.
     \end{align*}
 Since in the linear part we have convergence,  the proof is complete.

\end{proof}



\begin{theorem}\label{th:existH}
Let $\omega_0$ in~$H^{1,\Phi(\cdot)}(\Omega, \Lambda^{k}(\setR^N))$. Then the problem \eqref{VarProbl1} has a minimizer $\omega \in \omega_0 +H^{d,\Phi(\cdot)}_{ T}(\Omega,\Lambda^{k})$ with $\delta \omega=0$.
\end{theorem}

\begin{proof}
We  keep the notation from the proof of Theorem~\ref{th:existW}. Let~$\omega_s=\omega_0+\gamma_s$, $\gamma_s\in C_0^\infty(\Omega, \Lambda^k)$, be a minimizing sequence, clearly  
$$
    \norm{d\omega_s}_{L^{\Phi(\cdot)}(\Omega, \Lambda^{k+1})},\ \norm{d\gamma_s}_{L^{\Phi(\cdot)}(\Omega, \Lambda^{k+1})}\le c.
$$    
Due to the coercitivity condition~\eqref{eq:growth} we have~
$$
    \norm{d\omega_s}_{L^{p_{-}}(\Omega, \Lambda^{k+1})},\ \norm{d\gamma_s}_{L^{p_{-}}(\Omega, \Lambda^{k+1})}\le c.
$$
By Corollary~\ref{Forms:existence1} there exists~$\alpha_s\in \omega_0+ W^{1,p_{-}}_{T}(\Omega, \Lambda^{k})$ satisfying $d\alpha_s=d\omega_s$ and $\delta \alpha_s=0$ in $\Omega$
such that
$$
\norm{\alpha_s}_{W^{1,p_{-}}(\Omega,\Lambda^{k})}\leq c.
$$
Writing $\alpha_s = \omega_0+\beta_s$, one gets $\beta_s \in W^{1,p_{-}}_{T}(\Omega, \Lambda^{k})$ satisfying $d\beta_s = d \gamma_s$, $\delta \beta_s = - \delta \omega_0$. Extend $\beta_s$ to $\mathbb{R}^N\setminus \Omega$ by zero.

Clearly~$\alpha_s\in W^{d,\Phi(\cdot)}(\Omega,\Lambda^{k})$  and
$$
    \norm{\alpha_s}_{W^{d,\Phi(\cdot)}(\Omega,\Lambda^{k})},\ \norm{\beta_s}_{W^{d,\Phi(\cdot)}(\Omega,\Lambda^{k})} \le c.
$$
    
Let $\varphi_t:\Omega\times [0,1]\to\mathbb{R}^N$, $t\in (0,1)$, be $C^2$ mapping such that $\varphi_1=\mathrm{id}$, and $\varphi_t^{-1}\Omega \Subset \Omega$. If $\Omega$ is a ball centered at the origin, one takes $\varphi_t(x) = x/t$. Consider the pullbacks $\varphi_t^* \beta_s$. These forms have compact support in $\Omega$, with $d \varphi_t^* \beta_s$ uniformly converging to $d\beta_s=d\gamma_s\in C_0^\infty(\Omega,\Lambda^{k+1})$ and $\delta \varphi_t^* \beta_s$ converging to $-\delta \omega_0$ in $L^{p_{-}}(\Omega,\Lambda^{k-1})$ as $t\to 1-0$. Moreover, $\varphi_t^* \beta_s$ converges to $\beta_s$ in $W^{1,p_{-}}(\Omega,\Lambda^{k-1})$ as $t\to 1-0$.

Mollifications $(\varphi_t^* \beta_s)_\varepsilon = \chi_\varepsilon\ast \varphi_t^* \beta_s$, where $\chi_\varepsilon(x) = \varepsilon^{-d} \chi(x/\varepsilon)$, $\chi \in C_0^\infty(\{|x|<1\})$ with $\int \chi \, dx =1$,  converge to $\varphi_t^* \beta_s$ in $L^{p_{-}}(\Omega,\Lambda^k)$, $d (\varphi_t^* \beta_s)_\varepsilon$ converges uniformly to $d\varphi_t^* \beta_s$, $(\varphi_t^* \beta_s)_\varepsilon\to \varphi_t^* \beta_s$ in $W^{1,p_{-}}(\Omega,\Lambda^k)$, and $\delta  (\varphi_t^* \beta_s)_\varepsilon\to \delta \varphi_t^* \beta_s $ in $L^{p_{-}}(\Omega,\Lambda^{k-1})$ as $\varepsilon \to 0$. Clearly, $(\varphi_t^* \beta_s)_\varepsilon \in C_0^\infty(\Omega,\Lambda^k)$ for sufficiently small $\varepsilon$.

Therefore, keeping the same notation for $\beta_s$ while replacing it by $(\varphi_t^* \beta_s)_\varepsilon$ for appropriate $t$ and $\varepsilon$, we can assume that the new minimizing sequence has the form $\alpha_s = \omega_0 + \beta_s$, where $\beta_s\in C_0^\infty(\Omega,\Lambda^k)$, $\beta_s$ is uniformly bounded in $W^{1,p_{-}}(\Omega,\Lambda^k)$ and in $W^{d,\Phi(\cdot)}(\Omega,\Lambda^k)$, and $\|\delta (\omega_0+\beta_s)\|_{L^{p_{-}}(\Omega,\Lambda^{k-1})} < 1/s$.


Therefore there exists
$$
\beta\in W^{d,\Phi(\cdot)}_{T}(\Omega, \Lambda^{k}) \cap L^{p_{-}}(\Omega,\Lambda^{k})
$$
such that $\delta(\omega_0+\beta)=0$ and up to the subsequence,
\begin{gather*}
\beta_s\to \beta\quad \text{in }\quad W^{1,p_{-}}(\Omega,\Lambda^{k}),\\
d\beta_s \rightharpoonup d\beta \quad \text{in }\quad L^{\Phi(\cdot)}(\Omega,\Lambda^{k+1}). 
\end{gather*}
    Due to the convexity of~$\Phi(x,\cdot)$ and Mazur's lemma,  we have $\beta \in H_T^{1,\Phi(\cdot)}(\Omega,\Lambda^k)$ and for $\alpha=\omega_0+\beta$ there holds
        \begin{align*}
            \liminf_{s\to \infty} \int\limits_\Omega \Phi(x,|d\alpha_s|)\, dV \ge \int\limits_\Omega \Phi(x,|d\alpha|)\, dV.
        \end{align*}
 Since in the linear part we have convergence,  the proof is complete.

\end{proof}


\section{ Lavrentiev gap and non-density} \label{sec:framework}

In this section we design the general framework for the construction of the examples on Lavrentiev gap. We introduce the set of assumptions for the examples in the Section~\ref{sub:ass} and show how to obtain non-density of smooth functions and the special type of the non-uniqueness of the minimisers  under these assumptions. 
In Section \ref{sec:Basic} we introduce basic forms which will be building blocks of our examples. These building blocks correspond to the one saddle-point geometry of the classical checkerboard   Zhikov example and are then used in Sections  \ref{sec:Cantor} and \ref{sect:frac} to construct more advanced examples using fractal Cantor barriers.  The results are summarised in the subsection~\ref{sub:work}. 

\subsection{Separating pairs of forms and separating functionals}\label{sub:ass}
Here we present some ``conditional'' statements. We shall use two assumptions. Let $\Omega$ be a domain in $\mathbb{R}^N$ with sufficiently regular boundary, $k\in \{1,\ldots,N-1\}$, and $\frS \subset \Omega$ be a closed set of zero Lebesgue $N$-measure. Our argument will be based upon defining a suitable set $\frS$ and $(k-1)$-form $u$ and $(N-k-1)$-form $A$, which are smooth in $\Omega\setminus \frS$ and give a  ``counterexample'' to the Stokes theorem.  The regularity of $\partial \Omega$ is assumed to be such that the classical Stokes theorem holds. Further $\Omega$ will be either cube of ball in $\mathbb{R}^N$.


Let $\Phi:\Omega \times [0,+\infty)\to [0,+\infty)$ be a generalized Orlicz function.

\begin{definition}
 We say that a pair of $(k-1)$-form and $(N-k-1)$-form $(u,A)$ defined in $\Omega$ is $(\Phi,k)$-separating if there exists a closed set $\frS\subset \Omega$ of zero Lebesgue $N$-measure such that

\begin{enumerate}[label=(\roman*)]
\item $u$ and $A$ are regular outside $\frS$;

\item  $u\in W^{d,1} (\Omega, \Lambda^{k-1})$ and $A\in W^{d,1}(\Omega, \Lambda^{N-k-1})$;

\item $\int\limits_{\partial \Omega} A\wedge du =1$;

\item $|du|\cdot |dA|=0 \quad \text{in }\quad \Omega\setminus \frS$;

\item  $\int\limits_\Omega \Phi(x,|du|)\, dV <\infty$ and  $\int\limits_\Omega \Phi^*(x,|dA|)\, dV< \infty$.
\end{enumerate}
When invoking a pair of $(\Phi,k)$-separating forms we assume that the set $\frS$ comes from this definition and when necessary denote it by $\frS (u,A)$.
\end{definition}

The essential property of $(\Phi,k)$-separating forms is that $A\wedge du$ ``contradicts'' the Stokes theorem. Indeed, disregarding the singular set $\Sigma$ we would arrive at
$$
0=\int\limits_\Omega dA\wedge du = \int\limits_\Omega d(A\wedge du)= \int\limits_{\partial \Omega} A \wedge du=1.
$$


\begin{definition}
  \label{def:localuAb}
 Let $u$ and $A$ be a pair of $(\Phi,k)$-separating forms and $\eta \in C^\infty_0(\Omega)$ with $\eta=1$ in a neighbourhood of $\frS$. 
Set
  \begin{align*}
    u^\circ &= \eta u, &  u^\partial &= (1-\eta) u,
    \\
    A^\circ &= \eta A, &  A^\partial &= (1-\eta) A.
  \end{align*}
On~$W^{d,\Phi(\cdot)}(\Omega,\Lambda^{k-1})$ we define the functionals $\mathcal S$, $\mathcal S^\circ$, and $\mathcal S^\partial$ by  
$$
\mathcal S(w):=\int\limits_\Omega dA\wedge d w, \quad
\mathcal S^\circ(w):= \int\limits_\Omega dA^\circ \wedge d w, \quad 
\mathcal S^\partial(w):= \int\limits_\Omega dA^\partial \wedge d w. 
$$
\end{definition}

\begin{proposition}[Separating functional] \label{prop:separ}
The following holds
\begin{enumerate}

\item $\mathcal S, \mathcal S^\circ, \mathcal S^\partial  $ define linear functionals on~$W^{d,\Phi(\cdot)}(\Omega,\Lambda^{k-1})$. 

 \item For all~$w\in H^{d,\Phi(\cdot)}(\Omega,\Lambda^{k-1})$ we have $ S^\circ(w)=0$.
\item For the functions~$u,u^\partial, u^\circ$ it holds:
\begin{align*}
 \mathcal{S}(u) &= 0,\quad \mathcal{S}(u^\partial) = 1,\quad \mathcal{S}(u^\circ) =
    -1,\\
 \mathcal{S^\partial}(u)& = 1,\quad 
    \mathcal{S^\partial}(u^\partial) = 1,\quad 
    \mathcal{S^\partial}(u^\circ) = 0,\\
    \mathcal{S}^\circ(u) &= -1,\quad  \mathcal{S}^\circ(u^\partial) =
    0,\quad \mathcal{S}^\circ(u^\circ) = 
    -1.  
\end{align*}

\end{enumerate}

 \end{proposition}

\begin{proof}

The first claim follows from $dA \in L^{\Phi^*}(\Omega,\Lambda^{N-k})$. Due to the Stokes theorem and by approximation for all~$\omega \in H^{d,\Phi(\cdot)}(\Omega,\Lambda^{k-1})$ it holds 
  \begin{align*}
    \int\limits_\Omega dA^\circ\wedge d\omega=\int\limits_{\partial \Omega}A^\circ \wedge d\omega=0.
  \end{align*}
Now   $du \wedge  d A=0$ almost everywhere, therefore $\mathcal{S}(u)=0$.
  Since
$u^\partial\in C^\infty(\overline{\Omega},\Lambda^{k-1})$ and $A^\partial \in C^\infty(\overline{\Omega},\Lambda^{N-k-1}(\mathbb{R}^{N}))$, we can use the Stokes theorem and  the third property of $(\Phi,k)$-separating pair to obtain
  \begin{align*}
    \mathcal{S}^\partial(u^\partial) &= \int\limits_\Omega dA ^\partial\wedge du^\partial 
    = \int\limits_{\partial \Omega} A^\partial \wedge du^\partial
      = \int\limits_{\partial \Omega} A \wedge du = 1.
  \end{align*}
Since $A^\circ \wedge du^\partial$ belongs to $C_0^\infty(\Omega , \Lambda^{N-1})$, and $d (A^\circ \wedge du^\partial)=  dA^\circ \wedge d u^\partial$, again by the Stokes theorem we get
  \begin{align*}
    \mathcal{S}^\circ(u^\partial) = \int_\Omega dA^\circ \wedge d u^\partial
      = 0. 
  \end{align*}
  Analogously, we obtain $\mathcal{S}^\partial(u^\circ)=0$. Now,
  \begin{align*}
    \mathcal{S}^\circ(u^\circ) &= 
    \mathcal{S}(u) -
    \mathcal{S}^\partial(u^\partial) -
    \mathcal{S}^\circ(u^\partial) - 
    \mathcal{S}^\partial(u^\circ) = 0 - 1 - 0 - 0 = -1.
  \end{align*}
  This proves the claim.
\end{proof}

\begin{corollary}\label{corr:nondense1}
If there exists a pair of $(\Phi,k)$-separating forms then 
$$
H^{d,\Phi(\cdot)}(\Omega,\Lambda^{k-1}) \neq W^{d,\Phi(\cdot)}(\Omega,\Lambda^{k-1}).
$$ 
\end{corollary}
\begin{proof}
By Proposition~\ref{prop:separ}, $\mathcal{S}^\circ=0$ on $H^{d,\Phi(\cdot)}(\Omega,\Lambda^{k-1})$. On the other hand, $u\in W^{d,\Phi(\cdot)}(\Omega,\Lambda^{k-1})$ and $\mathcal{S}^\circ (u) = -1$.
\end{proof}

\begin{corollary}\label{corr:nondense2}
If there exists a pair of $(\Phi,k)$-separating forms then 
$$
H^{d,\Phi(\cdot)}_T(\Omega,\Lambda^{k-1}) \neq \tilde W^{d,\Phi(\cdot)}_T(\Omega,\Lambda^{k-1}).
$$  
\end{corollary}
\begin{proof}
For any $\varphi \in C_0^\infty(\Omega,\Lambda^{k-1})$ by the Stokes theorem we have 
$$
\mathcal{S} (\varphi) = \int\limits_\Omega dA \wedge d\varphi = \int\limits_{\partial \Omega} A \wedge d\varphi =0.
$$
On the other hand, $u^\circ\in \tilde W^{d,\Phi(\cdot)}_T(\Omega,\Lambda^{k-1})$ and by Proposition~\ref{prop:separ} we have $\mathcal{S}(u^\circ)=-1$.
\end{proof}

\begin{theorem} [Lavrentiev gap]\label{main}
If there exists a pair $(u,A)$ of $(\Phi,k)$-separating forms then for $b=dA^\circ$ the functional  
$$
\mathcal{F}_{\Phi,b}(w) = \int\limits_\Omega \Phi(x,|d w|)\,dV + \mathcal{S}^\circ(w)=  \int\limits_\Omega \Phi(x,|d w|)\,dV +\int\limits_\Omega b \wedge dw
$$
satisfies
$$
 \inf  \mathcal{F}(W^{d,\Phi(\cdot)}_c(\Omega,\Lambda^{k-1})) < \inf  \mathcal{F}(C_0^\infty (\Omega,\Lambda^{k-1}))
$$
and as a corollary
$$
 \inf  \mathcal{F}(W^{d,\Phi(\cdot)}_T(\Omega,\Lambda^{k-1})) < \inf  \mathcal{F}(H^{d,\Phi(\cdot)}_T (\Omega,\Lambda^{k-1})).
$$

\end{theorem}

\begin{proof}
By Proposition~\ref{prop:separ} and nonnegativity of $\Phi$, $\mathcal{F}_{\Phi,b}(w) \geq 0$ for all $w\in H^{d,\Phi(\cdot)}_T (\Omega,\Lambda^{k-1})$. On the other hand, for $t>0$, using Proposition~\ref{prop:separ} and \eqref{eq:phi-Nabla2}, we have
$$
\mathcal{F}_{\Phi,b}(tu^\circ) = \int\limits_\Omega \Phi(x,t|du^\circ|)\, dV -t \leq c t^{s} -t
$$
 with some $s>1$. This implies $\mathcal{F}_{\Phi,b}(tu^\circ) <0$ for sufficiently small $t$. 
\end{proof}

Now we discuss the Dirichlet problem. First, we repeat certain result from \cite{BalDieSur20}. Let 
$$
\mathcal{F} (\omega) = \int\limits_\Omega \Phi(x,|d\omega|)\, dV, \quad \mathcal{F}^*(g) := \int_\Omega \Phi^*(x , \abs{g(x)})\,dV. 
$$
Let $(u,A)$ be a  $(\Phi,k)$-separating pair. Denote $b=dA$.

\begin{assumption} \label{Ass:basic1} There exist $s,t > 0$ such that $\mathcal{F}(tu) + \mathcal{F}^*(s b) < t s$.
\end{assumption}

\begin{theorem}[H-harmonic $\neq$ W-harmonic]
  \label{thm:harmonic}
  
Under Assumption \ref{Ass:basic1}, for
  \begin{align*}
    w_t &= \argmin
          \mathcal{F}\big( tu^\partial+ W_0^{1,\Phi(\cdot)}(\Omega,\Lambda^{k-1})\big) \qquad
    \\
    h_t &= \argmin \mathcal{F}\big( tu^\partial+ C_0^{\infty}(\Omega,\Lambda^{k-1})\big)
  \end{align*}
  we have $w_t \neq h_t$ and $\mathcal{F}(w_t) < \mathcal{F}(h_t)$.
\end{theorem}
\begin{proof}
Set $b=dA$.  We have  $t u = tu^\partial + tu^\circ \in tu^\partial+ W_0^{1,\Phi(\cdot)}(\Omega,\Lambda^{k-1})$. Thus,
  \begin{align}
    \label{eq:Fwt}
    \mathcal{F}(w_t) &\leq \mathcal{F}(t u).
  \end{align}
  By the properties of the Hodge dual and the Young inequality, 
$$
   s b\wedge dh_t = s(* b, dh_t) \leq s|*b|\cdot |dh_t| = s|b| \cdot |dh_t| \leq \Phi(x,\abs{d h_t})+\Phi^*(x,s \abs{b}). 
$$
  
 Hence
$$    
\mathcal{F}(h_t)= \int_\Omega \Phi(x,\abs{d h_t})\,dx \geq s\int_\Omega b \wedge d h_t  - \int_\Omega\Phi^*(x,s \abs{b})\,dx= s\,\mathcal{S}(h_t) - \mathcal{F}^*(s b).
$$
See \cite{IwaKauKra04} for estimates of exterior product submultiplication constant.  
  
  Since $h_t - tu^\partial \in H_0^{d,\Phi(\cdot)}(\Omega)$, we have
  $\mathcal{S}(h_t - tu^\partial)=0$ by
  Proposition~\ref{prop:separ}. This and
  $\mathcal{S}(u^\partial)=1$ by the same Proposition imply
  \begin{align}
    \label{eq:Fht}
    \mathcal{F}(h_t)
    &= s\,\mathcal{S}(t u^\partial) - \mathcal{F}^*(s b)
    = t s - \mathcal{F}^*(s b).
  \end{align}
  Combining~\eqref{eq:Fwt} and~\eqref{eq:Fht} we get
  \begin{align*}
    \mathcal{F}(h_t) - \mathcal{F}(w_t) \geq t s -
    \mathcal{F}(tu) - \mathcal{F}^*(s b)
  \end{align*}
  for all $t,s > 0$. By Assumption~\ref{Ass:basic1} the right hand-side of last inequality is positive, and thus $\mathcal{F}(h_t) > \mathcal{F}(w_t)$. This proves the claim.
\end{proof}




\subsection{Basic forms}\label{sec:Basic}

In this section we introduce differential forms which will be building blocks of our examples.  We do necessary calculations in the cubic setting, where the boundary orientation is straightforward.

Let $k\in \{1,\ldots,N-1\}$.  Define two groups of variables $\bar x = (x_1,\ldots,x_k)$ and $\hat x = (x_{k+1},\ldots, x_N)$. Let $\Gamma_l(x)$, $x\in \mathbb{R}^l$, denote the fundamental solution of the Laplace equation in $\mathbb{R}^l$ with pole at the origin:
$$
\Gamma_l(x)= \begin{cases}
\frac{1}{2} |x|, \quad &l=1,\\
-\frac{1}{2\pi} \ln \frac{1}{|x|}, \quad &l=2,\\
-\frac{1}{(l-2)\sigma_l} |x|^{2-l}, \quad &l> 2.
\end{cases}
$$
Here and below $\sigma_l$ denotes the surface area ($(l-1)$-volume) of the unit sphere in $\mathbb{R}^l$, and $|x|$ denotes the standard Euclidian norm of $x$.

Let $\theta:\mathbb{R}\to \mathbb{R}$ be a smooth increasing function such that 
$$
\theta(t) =1\quad \text{for}\quad t \geq \frac{1}{2},\quad \theta(t)=0 \quad \text{for}\quad t\leq \frac{1}{4},\quad |\theta'| \leq 4.
$$ 
Let $\eta:\mathbb{R}\to \mathbb{R}$ be a smooth increasing function such that 
$$
\eta(t)=t \quad \text{for}\quad t\leq \frac{1}{4},\quad \eta(t) = \frac{1}{2} \quad \text{for}\quad t\geq \frac{3}{4},\quad \eta''(t)\leq 0.
$$

Our basic forms are 
\begin{gather}\label{base_u}
u = \theta \left(\sqrt{N}\frac{|\hat x|}{ \eta(|\bar x|)} \right) \ast_{\hat x} d \Gamma_{N-k} (\hat x) , \\
A =  \theta \left(\sqrt{N}\frac{|\bar x|}{\eta(|\hat x|)} \right) \ast_{\bar x} d \Gamma_{k} (\bar x), \label{base_A}
\end{gather}
Here $\ast_{\hat x}$ and $\ast_{\bar x}$ are applied only within respective variables, that is
\begin{gather*}
\ast_{\hat x} d \Gamma_{N-k} (\hat x) =\frac{1}{\sigma_{N-k}} \sum_{j=k+1}^N (-1)^{j-k-1} \frac{x_j} {|\hat x|^{N-k}} dx_{k+1} \wedge \ldots \wedge\widehat{d x_j} \wedge\ldots \wedge d x_N,\\
\ast_{\bar x} d \Gamma_{k} (\bar x) =\frac{1}{\sigma_k} \sum_{j=1}^k  (-1)^{j-1} \frac{x_j} {|\bar x|^{k}}dx_{1} \wedge \ldots \wedge\widehat{d x_j} \wedge\ldots \wedge d x_k.
\end{gather*}

Further for $(N-k-1)$-form $u$ from \eqref{base_u} and $(k-1)$-form $A$ from \eqref{base_A}we use the notation $u= \mathcal{P}_1(k,N-k,0,0)$ and $A=\mathcal{P}_2(k,N-k,0,0)$. Also, in this case we denote $\frC=\{0\}^k\subset \mathbb{R}^k$, $\frS =\{0\}^k \times \{0\}^{N-k}\subset \mathbb{R}^N$, and this pair of forms is denoted by $u_\frS$, $A_{\frS}$. 



The following facts are straightforward.

\begin{proposition}\label{prop:Gamma}
Both $\ast_{\hat x} d \Gamma_{N-k} (\hat x) $ and $\ast_{\bar x} d \Gamma_{k} (\bar x)$ are harmonic
\begin{gather*}
d (\ast_{\hat x} d \Gamma_{N-k} (\hat x) ) =0, \quad \delta (\ast_{\hat x} d \Gamma_{N-k} (\hat x))  =0,\\
 d( \ast_{\bar x} d \Gamma_{k} (\bar x)) =0, \quad \delta (\ast_{\bar x} d \Gamma_{k} (\bar x)) =0
\end{gather*}
outside $\hat x=0$ and $ \bar x=0$ correspondingly. For cubes $(-\varepsilon, \varepsilon)^{N-k}\subset \mathbb{R}^{N-k}$ and  $(-\varepsilon, \varepsilon)^{k}\subset \mathbb{R}^{k}$, $\varepsilon>0$, there holds
$$
\int\limits_{\partial(-\varepsilon, \varepsilon)^{N-k}} \ast_{\hat x} d \Gamma_{N-k} (\hat x) = 1, \quad \int\limits_{\partial(-\varepsilon, \varepsilon)^{k}} \ast_{\bar x} d \Gamma_{k} (\bar x)= 1,
$$
where the natural induced orientations of the boundary are assumed.

\end{proposition}


\begin{proposition}
For the forms $u$ and $A$ given by \eqref{base_u} and \eqref{base_A} 
\begin{enumerate}
\item  There holds
$$
\{u\neq 0\} \subset \{|\hat x| > \eta(|\bar x|)/(4\sqrt{N})  \}, \quad \{A\neq 0\} \subset \{|\bar x| > \eta(|\hat x|)/(4\sqrt{N})  \}.
$$

\item The forms $u$ and $A$ are smooth outside the origin, 
\begin{equation}\label{eq:prop1}
\begin{gathered}
|\nabla u|\lesssim |\hat x|^{k-N},\quad |\nabla A|\lesssim |\bar x|^{-k},\\
\{|\nabla u|\neq 0\} \subset \{|\hat x| > \eta(|\bar x|)/(4\sqrt{N})  \},\\
\{|\nabla A|\neq 0\} \subset \{|\bar x| > \eta(|\hat x|)/(4\sqrt{N})  \}.
\end{gathered}
\end{equation}

For any bounded domain $\Omega\subset \mathbb{R}^N$ there holds $u\in W^{1,1}(\Omega, \Lambda^{N-k-1})$, $A \in W^{1,1}(\Omega, \Lambda^{k-1})$, and 
$$
d u = d \theta \left(\sqrt{N}\frac{|\hat x|}{ \eta(|\bar x|)} \right) \wedge \ast_{\hat x} d \Gamma_{N-k} (\hat x), \quad dA = d \theta \left(\sqrt{N}\frac{|\bar x|}{ \eta(|\hat x|)} \right) \wedge \ast_{\bar x} d \Gamma_{k} (\bar x).
$$
\begin{equation}\label{eq:prop1a}
\begin{gathered}
|d u|\lesssim |\hat x|^{k-N},\quad |d A|\lesssim |\bar x|^{-k},\\
\{|d u|\neq 0\} \subset \{\eta(|\bar x|)/(2\sqrt{N})  >|\hat x| > \eta(|\bar x|)/(4\sqrt{N})  \},\\
\{|d A|\neq 0\} \subset \{\eta(|\hat x|)/(2\sqrt{N})>|\bar x| > \eta(|\hat x|)/(4\sqrt{N})  \}.
\end{gathered}
\end{equation}


\item There holds $|du| \cdot |dA| =0$ in $\mathbb{R}^N\setminus \{0\}$.

\item For a nonnegative function $F=F(\cdot,\cdot)$ with nonnegative arguments, satisfying $\triangle_2$--condition in the second argument and $F(\cdot,0)=0$,
\begin{equation}\label{du_est01}
\int\limits_{[-1,1]^N} F(|\hat x|, |d u|)\, dV \lesssim \int\limits_0^{\sqrt{N}} F\left(t, t^{k-N}\right) t^{N-1} \, dt.
\end{equation}
\end{enumerate}

\item For a nonnegative function $G=G(\cdot,\cdot)$ with nonnegative arguments, satisfying $\triangle_2$--condition in the second variable and $G(\cdot,0)=0$ 
\begin{equation}\label{b_est01}
\int\limits_{[-1,1]^N} G(|\hat x|, |dA|)\, dV \lesssim \int\limits_0^{\sqrt{N}} G\bigl(t, t^{-k} \bigr)\,  t^{N-1}  \, dt.
\end{equation}

\end{proposition}

\begin{proof}

The first two statements follow  from the definition of $u$ and $A$. Assume that $\Omega \subset \{|x|<R\}$. Using polar coordinates and estimates \eqref{eq:prop1}, we evaluate 
\begin{align*}
\int\limits_\Omega |\nabla u| \, dV &\lesssim \int\limits_0^R t^{k-N} t^{N-k-1} t^k\, dt = \int\limits_0^R t^{k-1}\, dt <\infty , \\
\int\limits_\Omega |\nabla A|\, dV &\lesssim \int\limits_0^R t^{-k} t^{k-1} t^{N-k}\, dt = \int\limits_0^R t^{N-k-1}\, dt <\infty.
\end{align*}
Thus the coefficients of the forms $u$ and $A$ belong to the Sobolev space $W^{1,1}(\Omega)$. Since the coefficients of the exterior derivative are linear combinations of derivatives of form coefficients, this implies $u\in W^{d,1}(\Omega, \Lambda^{N-k-1})$, $A\in  W^{d,1}(\Omega, \Lambda^{k-1})$ and their exterior derivatives are as above together with estimates \eqref{eq:prop1a}.

To prove that $|du| \cdot |dA| =0$ in $\mathbb{R}^N\setminus \{0\}$, we note that $|dA|\neq 0$ implies $|\bar x| < |\hat x|/(2\sqrt{N})$ and $du \neq 0$ implies $|\hat x|< |\bar x|/(2\sqrt{N})$ (recall that $\eta(t)\leq t$).

The last two statements immediately follows from the above estimates for $|\nabla u|$ and $|\nabla A|$ and using polar coordinates.
\end{proof}

Let $Q=[-1,1]^d$. For $x=(x_1,\ldots,x_l)\in \mathbb{R}^l$ the norm $|x|_{\infty} = \max\{|x_1|,\ldots,|x|_l\}$, while the standard Euclidian norm is denoted by $|x|=\sqrt{x_1^2+\ldots+ x_l^2}$. Recall that for $x\in \mathbb{R}^l$ there holds $|x|_\infty \leq |x|\leq \sqrt{l} |x|_\infty $.

\begin{proposition}\label{prop:boundary0}
For the form $A$ given by \eqref{base_A} on $\partial Q\cap \{|\hat x|_\infty<1\}$ there holds $dA=0$. Thus 
\begin{equation}\label{eq:inter1}
\begin{gathered}
\{dA\neq 0\}\cap \partial Q \subset \{|\hat x|_\infty=1\},\quad u=\ast_{\hat x} d \Gamma_{N-k} (\hat x) \quad \text{on}\quad \{dA\neq 0\} \cap \partial Q,\\
 dA = d\theta (2\sqrt{N}|\bar x|) \wedge \ast_{\bar x} d \Gamma_k(\bar x)  \quad \text{on}\quad \{dA\neq 0\} \cap \partial Q.
\end{gathered}
\end{equation}
\end{proposition}

\begin{proof}
Note that
$$
\{dA\neq 0\}\subset \{|\bar x| < \eta(|\hat x|)/(2\sqrt{N})\}\subset \{|\bar x| < |\hat x|/(2\sqrt{N})\} \subset \{|\hat x|_\infty>2|\bar x|_\infty\}.
$$
Then for $x\in \{dA\neq 0\} \cap \{|\bar x|_\infty=1\}$ there holds $|\hat x|_\infty>2$, which implies the first claim. Thus, 
\begin{equation*}
 \{dA\neq 0\} \cap \partial Q \subset \{|\hat x|_\infty =1\}\cap \{1/(8\sqrt{N})\leq |\bar x| \leq 1/(4\sqrt{N})\}. 
\end{equation*}
On the set $ \{dA\neq 0\} \cap \partial Q$ we have $|\hat x|_\infty=1$, $|\hat x| \geq 1$ and $|\bar x| \leq 1/4$, $\eta(|\bar x|) \leq 1/4$, so $\theta (\sqrt{N} |\hat x|/ \eta(|\bar x|)) =1$, $\eta (|\hat x|)=1$, $d\eta(|\hat x|) =0$, and we get \eqref{eq:inter1} by definitions \eqref{base_u} and \eqref{base_A} of $u$ and $A$.
\end{proof}

The following statement is central in our considerations. Let $\partial [-1,1]^N$ be the boundary of the cube $Q=[-1,1]^N$ with the natural induced orientation. 

\begin{lemma}\label{L:basic}
For the forms $u$ and $A$ given by \eqref{base_u} and \eqref{base_A} there holds
\begin{equation}\label{basic_integrals}
\int\limits_{\partial [-1,1]^N} u \wedge dA= (-1)^{k(N-k)}, \quad \int\limits_{\partial [-1,1]^N} A \wedge du =1.
\end{equation}
\end{lemma}
\begin{proof}

Below we use the notation of integration on cubic chains see \cite[Chapter 4]{Spivak}. Let 
$$
Q^l:\,[-1,1]^l\to \mathbb{R}^l,\quad Q^l(x)=x,
$$ 
be the standard $l$-cube and $\partial Q^l$ its boundary with the natural induced orientation.  

Denote the boundary faces of $Q^N$ as
\begin{align*}
&I_j^{\pm}:\, [-1,1]^{N-1} \to \mathbb{R}^N,\\ 
& I_j^{\pm} (x_1,\ldots,x_{j-1},x_{j+1},\ldots,x_N) = (x_1,\ldots,x_{j-1},\pm 1, x_{j+1},\ldots,x_{N}). 
\end{align*}
Then
$
\partial Q^N=\sum_{j=1}^N (-1)^{j-1} (I_j^{(+)}-I_j^{(-)}).
$

By $Q^{N-k}[\bar x]$ we denote the $(N-k)$-dimensional cubes  with centers at $(\bar x,0)$,
$$
Q^{N-k}[\bar x]: [-1,1]^{N-k}\to \mathbb{R}^{N},\quad Q^{N-k}[\bar x](x_{k+1},\ldots,x_N) = (\bar x,x_{k+1},\ldots,x_N).
$$
The faces $\tilde I_j^{(\pm)} [\bar x]$ of these cubes are
\begin{gather*}
\tilde I_j^{(\pm)} [\bar x] :\,[-1,1]^{N-k-1}\to \mathbb{R}^{N}, \quad \tilde I_j^{(\pm)} [\bar x] (x_{k+1},\ldots,x_{j-1},x_{j+1},\ldots,x_{N})\\
= (\bar x, x_{k+1}\ldots,x_{j-1},\pm 1,x_{j+1},\ldots,x_N),
\end{gather*}
and the boundary of $Q^{N-k}[\bar x]$ is
$
\partial Q^{N-k}[\bar x] = \sum_{j=k+1}^N  (-1)^{j-k-1} (\tilde I_j^{(+)} [\bar x]-\tilde I_j^{(-)} [\bar x]).
$

By \eqref{eq:inter1} we have 
\begin{align*}
&\int\limits_{\partial Q^N} u \wedge dA  =\int\limits_{\partial Q^N} \ast_{\hat x} d \Gamma_{N-k} (\hat x)  \wedge dA =  \\
&= 
\sum_{j=k+1}^N  (-1)^{j-1}  \biggl(\ \int\limits_{I_j^{(+)}} -\int\limits_{I_j^{(-)}} \biggr) \ast_{\hat x} d \Gamma_{N-k} (\hat x)  \wedge d \theta \bigl( 2\sqrt{N}|\bar x|\bigr) \wedge \ast_{\bar x} d \Gamma_{k} (\bar x) \\
&= (-1)^{(N-k)k} \sum_{j=k+1}^N  (-1)^{j-k-1}  \biggl(\ \int\limits_{I_j^{(+)}} -\int\limits_{I_j^{(-)}} \biggr)  d \theta \bigl( 2\sqrt{N}|\bar x|\bigr) \wedge \ast_{\bar x} d \Gamma_{k} (\bar x)   \wedge  \ast_{\hat x} d \Gamma_{N-k} (\hat x) \\
&= (-1)^{(N-k)k}\int\limits_{Q^k}  d( \theta (2\sqrt{N}|\bar x|) \ast_{\bar x} d \Gamma_{k} (\bar x) )    \sum_{j=k+1}^N  (-1)^{j-k-1}  \biggl(\ \int\limits_{\tilde I_j^{(+)} [\bar x]} -\int\limits_{\tilde I_j^{(-)} [\bar x]} \biggr)  \ast_{\hat x} d \Gamma_{N-k} (\hat x) \\
&=  (-1)^{(N-k)k}\int\limits_{Q^k} d( \theta (2\sqrt{N}|\bar x|) \ast_{\bar x} d \Gamma_{k} (\bar x) ) 
=  (-1)^{(N-k)k}\int\limits_{ \partial Q^k}  \theta (2\sqrt{N}|\bar x|) \ast_{\bar x} d \Gamma_{k} (\bar x) \\
&=  (-1)^{(N-k)k}\int\limits_{ \partial Q^k} \ast_{\bar x} d \Gamma_{k} (\bar x) =  (-1)^{(N-k)k}. 
\end{align*}
Thus we get the first relation in \eqref{basic_integrals}.
Since 
$$
A \wedge du - (-1)^{k(N-k)} u \wedge dA =d\omega,\quad \omega = (-1)^{(k-1)(N-k)} u \wedge A,
$$
we have 
$$
\int\limits_{\partial [-1,1]^N} \biggl( A \wedge du - (-1)^{k(N-k)} u \wedge dA \biggr)  =0.
$$
This yields the second relation in \eqref{basic_integrals}. The proof of Lemma~\ref{L:basic} is complete.
\end{proof}
 

To summarize the results of this section, we have shown that the pair of forms $u$ and $A$ given by \eqref{base_u} and \eqref{base_A} is $(\Phi,N-k)$-separating in $\Omega=[-1,1]^N$ provided that the integral \eqref{du_est01} converges for $F\geq\Phi$ and the integral \eqref{b_est01} converges for $G \geq \Phi^*$.

\subsection{Generalized Cantor sets and their properties.}\label{sec:Cantor}

In this section we construct (generalized) Cantor sets. 

Let $l_j$, $j=0,1,2,\ldots$ be a decreasing sequence of positive numbers starting from $l_0=1$:
$$
1=l_0>l_1>l_2>\ldots
$$
such that $l_{j-1} >2l_j$ for all $j\in \mathbb{N}$. We start from $I_{0,1}=[-1/2,1/2]$. On each $m$-th step we we remove the open middle third of length $l_j-2l_{j+1}$ from the interval $I_{m,j}$, $j=1,\ldots, 2^m$ to obtain the next generation set of closed intervals $I_{m+1,j}$, $j=1,\ldots, 2^{m+1}$. The union of the closed intervals  $I_{m,j}  = [a_{m,j}, b_{m,j}]$, $j=1,\ldots, 2^m$ of length $l_m$  from the same generation forms the pre-Cantor set $C_m = \bigcup_{j=1}^{2^m} I_{m,j}$. The Cantor set $\frC = \cap_{m=0}^\infty C_m$ is the intersection of all pre-Cantor sets $C_m$.  

On each $m$-th step  we define the pre-Cantor measure as $\mu_m = |C_m|^{-1}\mathbbm{1}_{C_m}$, where $|C_m|= 2^m l_m$ is the standard Lebesgue measure of $C_m$, and the weak limit of the measures $\mu_m$ is the Cantor measure corresponding to $\frC$.

We require further that
$$
l_{m-1} - 2l_m > l_m - 2 l_{m+1} \Leftrightarrow l_{m+1} > \frac{3l_m - l_{m-1}}{2}, 
$$
at least for all sufficiently large $m$.

If the sequence $l_j$ satisfies the conditions above only for sufficiently large $j\geq j_0$,  then we modify it by taking the sequence $\tilde l_j = l_{j+j_0} (l_{j_0})^{-1}$, $j=0,1,2,\ldots$

For $k\in \mathbb{N}$ by $\frC^k$ and $\mu^{k}$ we denote the Cartesian powers of $k$ copies of $\frC$ and its corresponding Cantor measure, respectively.


\bigskip

\noindent\textbf{Definition}. ({\it Generalized Cantor sets}). Let $l_j = \lambda^j j^\gamma$, $\lambda\in (0,1/2)$, $\gamma\in\mathbb{R}$. We denote the corresponding Cantor set by $\frC_{\lambda,\gamma}$, the Cartesian product of its $k$ copies is $\frC^k_{\lambda,\gamma}$. For $\frC^k_{\lambda,\gamma}$ we denote $\frD = -k\ln 2 / \ln \lambda$, so that $\lambda^{\frD} = 2^{-k}$. We denote the Cantor measure corresponding to $\frC_{\lambda,\gamma}$ by $\mu_{\lambda,\gamma}$ and its $k$-th Cartesian power by $\mu^k_{\lambda,\gamma}$.

\noindent\textbf{Definition}. ({\it Meager Cantor sets}).  Let $l_j = \exp(-2^{j/\gamma})$, $\gamma>0$. Denote  the corresponding Cantor set by $\frC_{0,\gamma}$, and its Cartesian products by $\frC^k_{0,\gamma}$. For these sets we denote $\frD=0$. We denote the corresponding Cantor measures by $\mu^k_{0,\gamma}$.

\bigskip

Denote $\frC_t = \{\mathrm{dist}(\bar x, \frC)<t\}$, where $\frC$ is one of $\frC_{\lambda,\gamma}^k$ or~$\frC^k_{0,\gamma}$ defined above. Denote  by  $d_\infty (\bar x,\frC)$ the distance from $\bar x$ to $\frC$ in the maximum norm and let $\frC_{*,t} = \{d_\infty(\bar x, \frC)<t\}$. It is clear that $\frC_t \subset \frC_{*,t}$. Let $|F|_k$ denote the standard Lebesgue $k$-measure of $F\subset \mathbb{R}^k$. In the following lemma $B^{\bar x}_t$ is the open ball in  $\mathbb{R}^k$ with center at $\bar x$ and radius $t$.


\begin{lemma}\label{L:me}
We have
\begin{equation}\label{cm1}
| (\frC^k_{\lambda,\gamma})_t|_k \lesssim t^{k-\frD} (\ln t^{-1})^{\gamma \frD}, \quad  \mu_{\lambda,\gamma}(B^{\bar x}_t)  \lesssim t^\frD (\ln t^{-1})^{-\gamma \frD},
\end{equation}
and 
\begin{equation}\label{cm2}
|(\frC^k_{0,\gamma})_t|_k \lesssim t^k(\ln t^{-1})^{\gamma {\color{red} k}} , \quad \mu_{0,\gamma} (B^{\bar x}_t)  \lesssim  (\ln t^{-1})^{-\gamma{\color{red} k}}.
\end{equation}
\end{lemma}

\begin{proof}
Let $l_j$ be the sequence of interval lengthes defining the corresponding Cantor set. Let $t \in (l_{j}/2 - l_{j+1}, l_{j-1}/2 - l_j)$. The set $\frC_{*,t}$ consists of $2^{kj}$ identical cubes of the form
$$
|\bar x - \bar x_{j,s}|_\infty  < \frac{l_j}{2} + t.
$$
So
$$
|\frC_{*,t}|_k\leq 2^{k(j-1)} (l_j + 2t)^k.
$$

First consider the case $\frC=\frC^k_{\lambda,\gamma}$ with $\lambda>0$. Then $l_{j}+2t \leq l_{j-1}-l_j \leq c t$ with some constant $c$ independent of $j$, so 
$$
|\frC_t|_k \leq |\frC_{*,t}|_k \lesssim 2^{kj} t^k.
$$
Recalling that $\frD = -k\ln 2 / \ln \lambda$ and $\lambda^{\frD} = 2^{-k}$, we get 
$$
2^{kj} = \lambda^{-j\frD}  = l_j^{-\frD} j^{\gamma \frD} \eqsim l_j^{-\frD} (\ln (1/l_j))^{\gamma \frD} \eqsim t^{-\frD} (\ln t^{-1})^{\gamma \frD}.
$$
Thus we arrive at the first inequality in \eqref{cm1}.

Now consider the case $\frC=\frC^k_{0,\gamma}$ (ultrathin Cantor sets). Then we get
$$
2^{kj}  = \left(\ln \frac{1}{l_j} \right)^{\gamma k} \approx  \left(\ln \frac{1}{t} \right)^{\gamma k}
$$
and this yields the second inequality from \eqref{cm1}.

Now let us estimate $\mu(B^{\bar x}_r)$. Any interval of length $2t$ with $t\in (l_{j}/2 - l_{j+1}, l_{j-1}/2 - l_j)$ can intersect at most one interval forming the $j$-th iteration of the pre-Cantor set. Since $B^{\bar x}_t$  lies within a cube with edge $2t$, then $\mu_{\lambda,\gamma}^k(B^{\bar x}_t) \leq 2^{-jk}$. Using the above estimates for $2^{jk}$ we arrive at \eqref{cm2}. The proof of Lemma~\ref{L:me} is complete.


\end{proof}

\subsection{From one singular point to fractal sets.}\label{sect:frac}

Let $k\in \{1,\ldots,N\}$, $\lambda \in (0,1/2)$ and $\gamma \in \mathbb{R}$, or $\lambda=0$ and $\gamma>0$ be given. Let $\Omega$ be the ball of radius $\sqrt{N}$ in $\mathbb{R}^N$ centered at the origin.

Now let $\frC=\frC^k_{\lambda,\gamma}$ be the generalized Cantor set with the given parameters and $\mu=\mu_{\lambda,\gamma}^k$ be the Cantor measure corresponding to $\frC^k_{\lambda,\gamma}$. Our construction will be based on the singular (or fractal contact/ barrier) set $\frS= \frC^k_{\lambda,\gamma} \times \{0\}^{N-k}$. Recall that for generalized Canter sets $\frC^k_{\lambda,\gamma}$ we set $\frD = -k \ln 2 / \ln \lambda$ (equivalently, $\lambda = 2^{-k/\frD}$) and for meager Cantor sets $\frC^k_{0,\gamma}$ we set $\frD=0$.

Let $d(\bar x,\frC)$ be the generalized distance, see \cite[Chapter VI, \S 2]{Stein16}  from $\bar x$ to $\frC$. In particular, $d(\bar x,\frC) \in C^\infty(\mathbb{R}^k \setminus \frC)$, 
\begin{equation}\label{dist_gen}
\frac{1}{C} \mathrm{dist} (\bar x,\frC) \leq d(\bar x,\frC) \leq C \mathrm{dist} (\bar x,\frC), \quad |\nabla d(\bar x,\frC) | \leq C,
\end{equation}
where $C>1$ and $ \mathrm{dist} (\bar x,\frC)$ is the standard Euclidian distance from $\bar x$ to $\frC$. Without loss, we assume that $C \geq 4$.

 Let $\theta:\mathbb{R}\to \mathbb{R}$ be a smooth nondecreasing function such that $\theta(t) =1$ for $t \geq 1/2$, $\theta(t)=0$ for $t\leq 1/4$, $|\theta'| \leq 4$. Let $\eta:\mathbb{R}\to \mathbb{R}$ be a smooth nondecreasing concave function such that $\eta(t)=t$ for $t\leq 1/4$ and $\eta(t) = 1/2$ for $t\geq 3/4$.

For $\frS = \frC \times \{0\}^{N-k}$, $\frC =\frC^k_{\lambda,\gamma}$, we define the $(N-k-1)$-form $u_\frS$, the $(k-1)$-form $A_\frS$, and the function $\rho_\frS$ by
\begin{align}\label{u_frac}
u_\frS &= \theta \left(\sqrt{N} C \frac{|\hat x|}{ \eta(d(\bar x,\frC))}\right) \ast_{\hat x} d \Gamma_{N-k}(\hat x), \\
A_\frS(\bar x, \hat x)& = \int A_\diamondsuit (\bar x-\bar y,\hat x) d\mu (\bar y), \quad A_\diamondsuit =  \theta \left(\sqrt{N}\frac{|\bar x|}{\eta(|\hat x|)} \right) \ast_{\bar x} d \Gamma_{k} (\bar x), \label{A_frac},\\
\label{def:rho}
\rho_\frS &= \theta \left(C\frac{|\hat x|}{3\eta(d(\bar x,\frC))}\right).
\end{align}
Here the constant $C$ is from \eqref{dist_gen}. The integral is understood as integrating the coefficients of the form. 

Further the $(N-k-1)$-form $u_\frS$ defined by \eqref{u_frac} and $(k-1)$-form $A_\frS$ defined by \eqref{A_frac} corresponding to the space dimension $N$ and the Cantor set $\frC_{\lambda,\gamma}$ will be denoted by $\mathcal{P}_1(k,N-k,\frD,\gamma)$ and $\mathcal{P}_2(k,N-k,\frD,\gamma)$. That is, $u_\frS= \mathcal{P}_1(k,N-k,\frD,\gamma)$ and $A_\frS=\mathcal{P}_2(k,N-k,\frD,\gamma)$. The function $\rho_\frS$ defined in \eqref{def:rho} will be also denoted by $\mathcal{P}_0(k,N-k,\frD,\gamma)$. 

\begin{lemma}\label{L:uest}
There holds $u\in W^{1,1}(\Omega,\Lambda^{N-k-1})\cap C^\infty(\overline\Omega\setminus \frS, \Lambda^{N-k-1})$ and
\begin{equation}\label{eq:u_est0}
\begin{gathered}
|\nabla u|(\bar x,\hat x) \lesssim \mathbbm{1}_{\{|\hat x|> \mathrm{dist}\,(\bar x,\frC)/(8C^2\sqrt{N})\}}  |\hat x|^{k-N},\\
|du|(\bar x,\hat x), |\delta u|(\bar x,\hat x) \lesssim \mathbbm{1}_{ \{ \mathrm{dist}\,(\bar x, \frC)/(8C^2\sqrt{N}) < |\hat x|< \mathrm{dist}\,(\bar x, \frC)/(2\sqrt{N})\}} |\hat x|^{k-N}.
\end{gathered}
\end{equation}
For a nonnegative function $F=F(\cdot,\cdot)$ with nonnegative arguments, satisfying $\triangle_2$--condition in the second argument and $F(\cdot,0)=0$,
\begin{equation}\label{du_est1}
\int\limits_\Omega F(|\hat x|, |\nabla u|)\, dV \lesssim \int\limits_0^{\sqrt{N}} F\left(t, t^{k-N}\right) t^{N-k-1} |\frC_t|_k \, dt.
\end{equation}
\end{lemma}

\begin{proof} Clearly, $u$ is smooth in $\Omega\setminus \frS$, and in particular its coefficients have ACL property.   Immediately from the definition of $u$, \eqref{dist_gen} and Proposition~\ref{prop:Gamma} we obtain \eqref{eq:u_est0}. Then using polar coordinates in $\mathbb{R}^{N-k}$ we obtain~\eqref{du_est1}. In particular, for $F(s,\tau)=\tau$ by Lemma~\ref{L:me} we get $|\nabla u|\in L^1(\Omega)$. By \cite[Theorem 2.1.4]{Zie89}  we conclude that the coefficients of $u$ are in $W^{1,1}(\Omega)$. 
\end{proof}

\begin{lemma}\label{L:Aest}
There holds $A\in W^{d,1}(\Omega,\Lambda^{k-1})\cap C^\infty(\overline\Omega\setminus \frS, \Lambda^{k-1})$,
$$
dA (\bar x ,\hat x)= \int b_\diamondsuit (\bar x-\bar y,\hat x) d\mu (\bar y), \quad \text{where}\quad b_\diamondsuit = dA_\diamondsuit,
$$
and 
\begin{equation}\label{eq:dAest0}
|dA|(\bar x,\hat x) \lesssim |\hat x|^{-k}  \mu (B^{\bar x}_{|\hat x|})  \mathbbm{1}_{\{\mathrm{dist}\,(\bar x,\frC) < |\hat x| /(2\sqrt{N})\}}(\bar x,\hat x).
\end{equation}
For a nonnegative function $G=G(\cdot,\cdot)$ with nonnegative arguments, satisfying $\triangle_2$--condition in the second variable and $G(\cdot,0)=0$, 
\begin{equation}\label{b_est1}
\int\limits_\Omega G(|\hat x|, |dA|)\, dV \lesssim \int\limits_0^{\sqrt{N}} G\bigl(t, t^{-k} \sup_{\bar x}\mu (B^{\bar x}_t \bigr) |\frC_t|_k\,  t^{N-k-1}  \, dt.
\end{equation}
\end{lemma}
\begin{proof}
Clearly, $A$ is smooth outside the contact set $\frS$. Denote 
$$
b(\bar x,\hat x) =  \int b_\diamondsuit (\bar x-\bar y,\hat x) d\mu (\bar y).
$$
Then
\begin{align*}
|b|(\bar x, \hat x) &\leq  \int |b_\diamondsuit| (\bar x-\bar y,\hat x) d\mu (\bar y) \lesssim \int |\hat x|^{-k}   \mathbbm{1}_{\{|\hat x|/(8\sqrt{N}) < |\bar x|< |\hat x|/(2\sqrt{N}) \}} (\bar x - \bar y,\bar x)  d\mu (\bar y) \\
&\lesssim  |\hat x|^{-k}  \int  \mathbbm{1}_{\{|\bar x|< |\hat x|/(2\sqrt{N}) \}} (\bar x - \bar y,\bar x)  d\mu (\bar y) \leq |\hat x|^{-k}  \mu (B^{\bar x}_{|\hat x|/(2\sqrt{N})})  \mathbbm{1}_{\{\mathrm{dist}\,(\bar x,\frC) < |\hat x| /(2\sqrt{N})\}}(\bar x,\hat x).
\end{align*}
Using polar coordinates in $\mathbb{R}^{N-k}$ we evaluate
\begin{gather*}
\int\limits_\Omega G(|\hat x|, |b|)\, dV \lesssim \int\limits_\Omega  G\bigl (|\hat x|, |\hat x|^{-k} \mu(B^{\bar x}_{|\hat x|}) \bigr)\mathbbm{1}_{\{\mathrm{dist}\,(\bar x,\frC) < |\hat x|\}}(\hat x,\bar x)  d\bar{x} d\hat{x} \\
\lesssim \int\limits_{\{|\hat x| \leq \sqrt{N}\}}    G\bigl(|\hat x|, |\hat x|^{-k} \sup_{\bar x}\mu(B^{\bar x}_{|\hat x|})\bigr) |\{\mathrm{dist}\,(\bar x,\frC) < |\hat x|\}|_{k}\,  d \hat{x}\\
=\int\limits_0^{\sqrt{N}} G\bigl(t, t^{-k} \sup_{\bar x}\mu (B^{\bar x}_t \bigr) |\frC_t|_k\,  t^{N-k-1}  \, dt.
\end{gather*}
In particular, using $G(s,t) = t$ and Lemma~\ref{L:me} we get $b\in L^1(\Omega)$.

For any $\varphi\in C_0^\infty (\Omega, \Lambda^k)$ using $A_\diamondsuit \in W^{1,1}(\Omega, \Lambda^k)$ we have
\begin{gather*}
(A, \delta \varphi) =  \int\limits_\Omega \biggl(\int A_\diamondsuit (\bar x-\bar y,\hat x) d\mu (\bar y) \biggr) \ \wedge \ast \delta \varphi   \\ 
=\int d\mu(y) \int\limits_\Omega  A_\diamondsuit (\bar x-\bar y,\hat x) \wedge \ast \delta \varphi  = \int d\mu (y) \int\limits_\Omega b_\diamondsuit   (\bar x-\bar y,\hat x) \wedge \ast \varphi = \int\limits_\Omega b \wedge \ast \varphi = (b,\varphi).
\end{gather*}
By definition, this implies $dA=b$, and as a consequence \eqref{eq:dAest0} and \eqref{b_est1}. 
The proof of Lemma~\ref{L:Aest} is complete.
\end{proof}

\begin{proposition} \label{dudA}
There holds \mbox{$|du| \cdot |dA| =0$} a.e. in $\Omega$.
\end{proposition}

\begin{proof}

By Lemmas~\ref{L:uest} and \ref{L:Aest}, 
$$
\{|du|> 0\} \subset \{|\hat x| \leq \mathrm{dist}\, (\bar x,\frC)  /(2\sqrt{N})\},\quad \{|dA|>0\} \subset\{|\hat x| \geq 2\sqrt{N}\, \mathrm{dist}\, (\bar x,\frC)\}.
$$ 
Thus $|du|\cdot |dA| = 0$ in $\overline{\Omega}\setminus \frS$. The claim follows since $\frS$ has $N$-dimensional Lebesgue measure zero.
\end{proof}


\begin{proposition}\label{prop:rho}

The function $\rho_\frS \in C^\infty(\mathbb{R}^N \setminus \frS)$, $0\leq\rho_\frS\leq 1$, $\rho_\frS = 1$ on the support of $dA$ and $\rho_\frS =0$ on the support of $du$.
\end{proposition}
\begin{proof}
The first two properties are immediate from the definition of $\rho_\frS$. From the definition of $u$, 
$$
\mathrm{supp}\, du \subset \{|\hat x| \leq \eta (d(\bar x,\frC)) / (\sqrt{N}C)\}.
$$ 
On this set, $\rho_\frS =0$ (recall that $\theta(t)=0$ if $t\leq 1/4$). On the other hand, 
$$
\mathrm{supp}\, dA \subset \{|\hat x| \geq 2\sqrt{N} \mathrm{dist}(\bar x,\frC)\} \subset \{|\hat x| \geq 2\sqrt{N} C^{-1} \eta(d(\bar x,\frC))\}.
$$
On this set, $\rho_\frS =1$ (recall that $\theta(t)=1$ if $t \geq 1/2$).
\end{proof}

\begin{proposition}\label{udA:bdry}
On the boundary of $Q=[-1,1]^N$ there holds
\begin{enumerate}

\item $u\wedge dA=0$  on $\partial [-1,1]^N \cap \{|\hat x|_\infty<1\}$;

\item On $\partial Q \cap \{|\hat x|_\infty=1\}$ there holds
$$
u=\ast_{\hat x} d \Gamma_{N-k},\quad dA = \int d (\theta (2\sqrt{N}|\bar x -\bar y|)) \wedge\ast_{\bar x} d \Gamma_k(\bar x-\bar y)\, d\mu(\bar y).
$$
\end{enumerate}

\end{proposition}

\begin{proof}

By construction, 
$$
\{|dA| \neq 0\} \subset \{ \mathrm{dist}(\bar x,\frC) < |\hat x|/(2\sqrt{N})\}\subset \{ \mathrm{dist}(\bar x,\frC) < |\hat x|_\infty/2\}.
$$
If $|\bar x|_\infty=1$, then $\mathrm{dist}(\bar x,\frC) \geq 1/2$ (recall that $\frC \subset [-1/2,1/2]^{k}$), so $dA (\bar x,\hat x) \neq 0$ implies $|\hat x|_\infty>1$. Thus 
$$
\{|dA| \neq 0\} \cap \partial Q \subset   \{|\hat x|_\infty=1\} \cap \partial Q.
$$
Then in the definition of $u$ for the argument of $\theta$ for $|\hat x|_\infty\leq 1$, $|\bar x|_\infty=1$ we have 
$$
\sqrt{N} C \frac{|\hat x|}{\eta (d(\bar x,\frC))}\geq \sqrt{N} \frac{|\hat x|}{\mathrm{dist}\,(\bar x,\frC)} \geq 2.
$$
This implies $\theta \left(\sqrt{N} C \frac{|\hat x|}{\eta (d(\bar x,\frC))} \right)=1$, and therefore $u = \ast_{\hat x} d \Gamma_{N-k}$ on $\partial [-1,1]^N \cap \{|\hat x|_\infty=1\}$.
 
The formula for $dA$ follows then from $\eta(|\hat x|)=1/2$ for $x$ in a neighbourhood of $[-1,1]^N \cap \{|\hat x|_\infty=1\}$ and smoothness of the integrand in the definition of $A$ for $|\hat x|>0$.
\end{proof}

\begin{lemma}\label{L:bi}
For the forms $u$ and $A$ given by \eqref{u_frac} and \eqref{A_frac} there holds
\begin{equation}\label{basic_integrals1}
\int\limits_{\partial [-1,1]^N} u \wedge dA= (-1)^{k(N-k)}, \quad \int\limits_{\partial [-1,1]^N} A \wedge du =1.
\end{equation}
\end{lemma}

\begin{proof}

Using Proposition~\ref{udA:bdry} and the notation of integration on cubic chains similar to Lemma~\ref{L:basic} we obtain
\begin{gather*} 
\int\limits_{\partial Q^N} u\wedge dA 
= \sum_{j=k+1}^N (-1)^{j-1} \biggl( \int\limits_{I^{+}_j}  - \int\limits_{I^{-}_j}\biggr) \ast_{\hat x} d \Gamma_{N-k}(\hat x) \wedge \int d (\theta (2\sqrt{N}|\bar x -\bar y|)) \ast_{\bar x} d\Gamma_k(\bar x - \bar y)) d\mu (\bar y) \\
= (-1)^{k(N-k)} \sum_{j=k+1}^N (-1)^{j-k-1} \biggl( \int\limits_{I^{+}_j}  - \int\limits_{I^{-}_j}\biggr) \biggl(d\int \theta(2\sqrt{N} |\bar x-\bar y|) \ast_{\bar x} d\Gamma_k(\bar x - \bar y)  d\mu (\bar y) \biggr) \wedge \ast_{\hat x} d \Gamma_{N-k}(\hat x)\\
=(-1)^{k(N-k)} \int \ \biggl[  \int\limits_{Q^k}d (\theta(2\sqrt{N} |\bar x-\bar y|) \ast_{\bar x} d\Gamma_k(\bar x - \bar y)) \sum_{j=k+1}^N (-1)^{j-k-1}\biggl( \int\limits_{\tilde I_j^{+}(\bar x)} - \int\limits_{\tilde I_j^{-}(\bar x)}  \biggr) \ast_{\hat x}d \Gamma_{N-k}(\hat x)\biggr] d\mu (\bar y)\\
=(-1)^{k(N-k)} \int d\mu(\bar y)= (-1)^{k(N-k)}.
\end{gather*}
Here we used that 
\begin{gather*}
 \int\limits_{Q^k}d (\theta(2\sqrt{N} |\bar x-\bar y|) \ast_{\bar x} d\Gamma_k(\bar x - \bar y)) \sum_{j=k+1}^N (-1)^{j-k-1}\biggl( \int\limits_{\tilde I_j^{+}(\bar x)} - \int\limits_{\tilde I_j^{-}(\bar x)}  \biggr) \ast_{\hat x}d \Gamma_{N-k}(\hat x)\\
 = \int\limits_{Q^k}d \bigl(\theta(2\sqrt{N} |\bar x-\bar y|) \ast_{\bar x} d\Gamma_k(\bar x - \bar y) \bigr)=
 \int\limits_{\partial Q^k}\theta(2\sqrt{N} |\bar x-\bar y|) \ast_{\bar x}d\Gamma_k(\bar x - \bar y)\\
 =\int\limits_{\partial Q^k} \ast_{\bar x} d\Gamma_k(\bar x - \bar y)=1
\end{gather*}
for all $y\in [-1/2,1/2]^k\subset \mathbb{R}^k$.

To calculate the second integral we use the same argument as in Lemma~\ref{L:basic}: the form $A\wedge du - (-1)^{k(N-k)} u \wedge dA$ is exact, therefore its integral over $\partial [-1,1]^N$ is zero.
\end{proof}

\subsection{Work-tool}\label{sub:work}

Here we gather the results of Section~\ref{sect:frac}, namely of Lemmas~\ref{L:uest}, \ref{L:Aest}, Propostion~\ref{dudA}, and Lemma~\ref{L:bi}. Recall that a pair of $(k-1)$-form $u$ and $(N-k-1)$-form $A$  is $(\Phi,k)$-separating if $u$ and $A$ are regular outside a closed set $\frS\subset \Omega$ of zero Lebesgue $N$-measure,  $u\in W^{d,\Phi(\cdot)} (\Omega, \Lambda^{k-1})$, $A\in W^{d,\Phi^*(\cdot)}(\Omega, \Lambda^{N-k-1})$, $|du|\cdot |dA|=0$ a.e. in $\Omega$, and $\int_\Omega A \wedge du =1$. The form of the following statement represents the duality between $u$ and $A$. The function $\Phi=\Phi(x,t)$ is a generalized Orlicz function, as in Section~\ref{sect:Orlicz}. The following lemma gives the general work-tool to construct a $(\Phi,k)$-separating pair. Let $\Omega=[-1,1]^N$.  

\begin{lemma}\label{lemma:integr}
\noindent (\textbf{i}) 
Let $u=\mathcal{P}_1(N-k,k,\frD,\gamma)$ and $A=\mathcal{P}_2(N-k,k,\frD,\gamma)$. Let $\Phi$ be such that $\Phi(x,t)\leq F_1(|\hat x|,t)$ on the support of $du$ and $\Phi(x,t)\geq F_2(|\hat x|,t)$ on the support of $dA$. If
\begin{equation}\label{c:sub}
\begin{gathered}
\mathcal{I}_1:=\int\limits_0^{\sqrt{N}} F_1(t,t^{-k})  |\frC_t|_{N-k} t^{k-1} \, dt<\infty,\\
\mathcal{I}_2:=\int\limits_0^{\sqrt{N}} F_2^* \bigl(t, t^{k-N}\sup_{\bar x} \mu (B^{\bar x}_t) \bigr)  |\frC_t|_{N-k}t^{k-1}\, dt<\infty,
\end{gathered}
\end{equation}
then the pair $(u,A)$ is $(\Phi,k)$-separating . 

\noindent(\textbf{ii}) 
Let $u=\mathcal{P}_2(k,N-k,\frD,\gamma)$ and $A=(-1)^{k(N-k)}\mathcal{P}_1(k,N-k,\frD,\gamma)$. Let $\Phi$ be such that $\Phi(x,t)\leq F_1(|\hat x|, t)$ on the support of $du$ and $\Phi(x,t)\geq F_2(|\hat x|,t)$ on the support of $dA$. If 
\begin{equation}\label{c:super}
\begin{gathered}
\mathcal{I}_1:=\int\limits_0^{\sqrt{N}} F_1 \bigl(t, t^{-k}\sup_{\bar x}\mu (B^{\bar x}_t ) \bigr) |\frC_t|_k\,  t^{N-k-1}  \, dt<\infty,\\
\mathcal{I}_2:=\int\limits_0^{\sqrt{N}} F_2^*  \bigl( t,t^{k-N} \bigr)  |\frC_t|_{k}t^{N-k-1}\, dt<\infty,
\end{gathered}
\end{equation}
then the pair $(u,A)$ is $(\Phi,k)$-separating. 

Moreover, in both cases there holds 
$$
 \int\limits_\Omega \Phi(x,|du|)\, dV\le C(N,k) \mathcal{I}_1,\quad \int\limits_\Omega \Phi^*(x,|dA|)\,dV\le C(N,k) \mathcal{I}_2.
$$

\end{lemma}

\begin{proof}

The forms $u$ and $A$ are regular outside $\frS$ by construction. By Lemmas~\ref{L:uest},\ref{L:Aest} we have $u\in W^{d,1}(\Omega,\Lambda^{k-1})$ and $A\in W^{d,1}(\Omega,\Lambda^{N-k-1})$.  By Proposition~\ref{dudA}, $|du|\cdot |dA| =0$ outside $\frS$. Since $d(A\wedge du) = dA \wedge du =0$ outside $\frS$, Lemma~\ref{L:bi} implies
$$
\int\limits_{\partial \Omega} A \wedge du = \int\limits_{\partial [-1,1]^N} A \wedge du=1.
$$
It only remains to prove that 
$$
\int_\Omega \Phi(x,|du|)\, dV<\infty, \quad \int_\Omega \Phi^*(x,|dA|)\, dV<\infty.
$$
But this follows from estimates \eqref{du_est1} and \eqref{b_est1} by the assumptions of the lemma.
\end{proof}

In view of Section~\ref{sec:framework}, to construct an example for the Lavrentiev gap, it is sufficient to check the conditions of Lemma~\ref{lemma:integr}. In the following section we do this for the ``standard'' and ``borderline'' double phase models and for the variable exponent. 

\subsection{Example setups}\label{sec:setup}

Two cases of Lemma~\ref{lemma:integr} and different choices of the fractal contact set give us several variants of example setup. Further $p_0>1$ will be the  threshold parameter. Depending on the value of the threshold parameter~$p_0$, we design 5 different setups:
\begin{enumerate}
\item critical or one saddle point setup corresponds to the classical Zhikov checkerboard example \cite{Zhi86} ($N=2$, $k=1$) and its development by \cite{EspLeoMin04} ( $N>1$, $k=1$); 

\item supercritical  setup corresponds to the case $p_0>N/k$, in the scalar settting ($k=1$) of \cite{BalDieSur20} this corresponds to the superdimensional case $p_0>N$ with singular set on a line; 

\item subcritical case corresponds to the case $1<p_0<N/k$, in the scalar settting ($k=1$) of \cite{BalDieSur20} this corresponds to the subdimensional case $1<p_0<N$ with singular set on a hyperplane; 

\item right limiting critical case corresponds to the situation when~$p_0=N/k +0$ (that is, for the critical value $p_0=N/k$ we use the supercritical construction); 

\item left limiting critical case corresponds to the situation when~$p_0= N/k-0 $ (that is, for the critical value $p_0=N/k$ we use the subcritical construction). 
\end{enumerate}

Each of these setups  includes the  fractal set~$\frC$ (see Section~\ref{sec:Cantor}, in the ``critical'' case it is just one point), the barrier fractal set~$\frS$, the pair of the forms~$u$ and~$A$, and the function~$\tilde{\rho}$ which separates the supports of $du$ and $dA$: it is equal to $0$ on the support of $du$ and $1$ on the support of $dA$. The construction of the forms $u$, $A$ and the function $\tilde \rho$ is described in \ref{sec:Basic} (for one singular point case) and in \ref{sect:frac} (for the rest of cases).

One can easily verify (this is done in Section~\ref{subsec:stan_d}) that $du \in L^{p}(\Omega, \Lambda^k)$ for any $p<p_0$ and $dA \in L^{q'}(\Omega,\Lambda^{N-k})$ for any $q>p_0$ which explains why we call this parameter ``threshold''. The function $\tilde \rho$ is then used to construct the function $\Phi$ for which the pair $(u,A)$ is $(\Phi,k)$-separating.

The second free parameter of the construction --- the shrinking fractal parameter~$\gamma$ --- plays an important role later in refining our examples to the limiting case and in treating the borderline double phase and the log-log-H\"older exponents.

\begin{enumerate}[label=Setup{ {\arabic*}}]
 
\item {\it({Critical or  one saddle point)}}  \label{set1}  Let $p_0= N/k$ and set
\begin{align*}
\frC&=\{0\}^{N-k},\quad \frS = \{0\}^N, &\tilde \rho= \rho_\frS= \mathcal{P}_0(N-k,k,0,0),\\
 u &= u_\frS=\mathcal{P}_1(N-k,k,0,0), &A =A_\frS=\mathcal{P}_2(N-k,k,0,0).
\end{align*}

\item\label{set2}  {\it (Supercritical)} Let $p_0>N/k$. Define $\frD=(p_0k-N)/(p_0-1)$ from $p_0 = (N-\frD) / (k-\frD)$ and  set $\lambda= 2^{-k/\frD}$,
\begin{align*}
\frC &= \frC^{k}_{\lambda,\gamma},\quad \frS = \frC\times \{0\}^{N-k}, &\tilde \rho= 1-\rho_\frS=1-\mathcal{P}_0(k,N-k,\frD,\gamma), \\  
u &= A_\frS= \mathcal{P}_2(k,N-k,\frD,\gamma), &A =(-1)^{k(N-k)}u_\frS= (-1)^{k(N-k)} \mathcal{P}_1(k,N-k,\frD,\gamma).
\end{align*}

\item \label{set3} {\it (Subcritical)} Let $1<p_0<N/k$. Define $\frD=N-p_0k$ from $p_0= (N-\frD)/k$ and set $\lambda = 2^{-(N-k)/\frD}$, 
\begin{align*}\frC&= \frC^{N-k}_{\lambda,\gamma},\quad \frS = \frC\times \{0\}^{k}, &\tilde \rho= \rho_\frS= \mathcal{P}_0(N-k,k,\frD,\gamma),\\ 
u&=u_\frS=\mathcal{P}_1(N-k,k,\frD,\gamma),&A =A_\frS= \mathcal{P}_2(N-k,k,\frD,\gamma).
\end{align*}
\item \label{set4}  {\it (Right limiting (critical+0))} Let $p_0=N/k$ and set
\begin{align*}
\frC &= \frC^{k}_{0,\gamma},\quad \frS = \frC^k_{0,\gamma}\times \{0\}^{N-k}, &\tilde \rho= 1-\rho_\frS=1-\mathcal{P}_0(k,N-k,0,\gamma),\\
u &= A_\frS= \mathcal{P}_2(k,N-k,0,\gamma),&A =(-1)^{k(N-k)}u_\frS= (-1)^{k(N-k)} \mathcal{P}_1(k,N-k,0,\gamma).
\end{align*}

\item \label{set5} {\it (Left limiting (critical-0))} Let $p_0=N/k$ and set
\begin{align*}
\frC&= \frC^{N-k}_{0,\gamma},\quad \frS = \frC\times \{0\}^{k},&\tilde \rho= \rho_\frS=\mathcal{P}_0(N-k,k,0,\gamma),\\
u&=u_\frS=\mathcal{P}_1(N-k,k,0,\gamma),&A = A_\frS= \mathcal{P}_2(N-k,k,0,\gamma).
\end{align*}

\end{enumerate}

\ref{set1}, \ref{set3}, \ref{set5}  correspond to Lemma~\ref{lemma:integr} ({\bf ii}), and \ref{set2} and  \ref{set4} correspond to Lemma~\ref{lemma:integr} ({\bf i}).

\section{Applications}
\label{sec:App}
In this section we show the presence of the Lavrentiev gap for the following models 
\begin{enumerate}
\item double phase;

\item borderline double phase;

\item variable exponent.
\end{enumerate}
 To this end we use the framework defined in Section~\ref{sec:framework} and the Cantor set-based construction from Section~\ref{sect:frac}. That is, we have to show that the pair of forms $u$ and $A$ build as in Section~\ref{sect:frac} is $(\Phi,k)$-separating and satisfies the conditions of  Assumption~\ref{Ass:basic1} (the latter one for the Dirichlet problem) for the generalized Orlicz functions
 \begin{enumerate}
\item $\Phi(x,t)=t^{p}+a(x)t^q$;

\item $\Phi(x,t)=t^p\log^{-\beta}(e+t)+a(x)t^p\log^\alpha(e+t)$;

\item $\Phi(x,t) = t^{p(x)}$.
\end{enumerate}

Further in this section $k\in \{1,\ldots,N-1\}$, $\frC=\frC^{l}_{\lambda,\gamma}$ is a generalized Cantor set as in Section~\ref{sec:Cantor}, and $\frS = \frC \times \{0\}^{N-l}$ is the singular contact set, where $l=k$ or $l=N-k$.  As above, by $\frC_t$ we denote the $t$-neighbourhood of the set $\frC$.

Recall that the parameter $\lambda$ of the fractal set $\frC^m_{\lambda,\gamma}$ is connected to its ``fractal dimension'' $\frD$ by $\frD = -m \ln 2/\ln \lambda$ if $\frD>0$ and $\lambda=0$ if $\frD=0$, and the forms $u$ and $A$ defined in \eqref{u_frac} and $\eqref{A_frac}$, based on the contact set $\frC^m_{\lambda,\gamma}$ (or \eqref{base_u} and \eqref{base_A} for $\frD=\gamma=0$) are denoted by $\mathcal{P}_1(m,N-m,\frD,\gamma)$ and $\mathcal{P}_2(m,N-m,\frD,\gamma)$. That is, the forms $\mathcal{P}_j(m,N-m,\frD,\gamma)$, $j=1,2$, together with the function $\mathcal{P}_0(m,N-m,\frD,\gamma)$ are constructed using the singular set $\frC^m_{\lambda,\gamma}\times \{0\}^{N-m}$ with $\lambda = 2^{-m/\frD}$ if $\frD>0$ and $\lambda=0$ if $\frD=0$. 



Before passing on to the examples we make the following observation.

\begin{lemma}\label{L:cont0}
 Let $a_0$ be an increasing concave function and $a_0(0)=0$. Let $\rho = \rho_\frS$ be the function defined by \eqref{def:rho}. Then the functions $a_0(|\hat x|) \rho(x)$ and $a_0(|\hat x|)(1-\rho(x))$ have the modulus of continuity $\tilde C a_0(\cdot)$ for some $\tilde C>1$. 
\end{lemma}
\begin{proof}

We this for the function $a_0(|\hat x|) \rho(x)$. First note that for $x=(\bar x,\hat x)$ there holds $|\nabla \rho|(x)\leq C |\hat x|^{-1}$ for some constant $C>1$. This follows straight from the definition of $\rho_\frS$ in \eqref{def:rho}. 

For $x=(\bar x,\hat x)$ and $y=(\bar y, \hat y)$ we evaluate 
\begin{gather*}
r(x,y):=|a_0(|\hat x|) \rho(x) - a_0(\hat y) \rho (y)| \leq |a_0(|\hat x|)-a_0(|\bar y|)|\rho(x) + a_0(|\hat y|) |\rho(x) - \rho(y)| \\
\leq a_0(|\hat x - \hat y|) + a_0(|\hat y|) |\rho(x) - \rho(y)|.
\end{gather*}

If $|x-y| \geq |\hat y|/2$ we evaluate $a_0(|\hat y|) \leq 2 a_0(|x-y|)$ using the concavity of $a_0$, and $|\rho(x)-\rho(y)|\leq 1$, therefore $r(x,y) \leq  3 a_0(|x-y|)$.

If $|x-y|\leq  |\hat y|/2$ then $|\hat y|/2\leq |\hat x| \leq 3 |\hat y|/2$, therefore $
|\rho(x)-\rho(y)| \leq 2C |\hat y|^{-1} |x-y|$. Now, 
$$
a_0(|\hat y|) |\rho(x)-\rho(y)| \leq 2C\frac{a_0(|\hat y|)}{|\hat y|} a_0(|x-y|) \frac{|x-y|}{a_0(|x-y|)} \leq 2C a_0(|x-y|) \leq a_0(|x-y|)
$$
since the concavity of $a_0$ implies that $a(s)s^{-1} \leq a(t) t^{-1}$ for $s\geq t$. Therefore, for $|x-y| \leq |\hat y|/4 $ we get $r(x,y) \leq (2C+1) a_0(|x-y|)$.

Thus in both cases we have $r(x,y) \leq 3C a_0(|x-y|)$.
\end{proof}

\subsection{Standard double phase model}
\label{subsec:stan_d} 

Let $1<p<q<+\infty$ and $\alpha\geq 0$, 
\begin{equation}\label{Orlicz1}
\varphi(t) = t^p, \quad \psi(t) = t^q.
\end{equation}
Denote 
\begin{equation}\label{adef}
a_0(t) = t^\alpha, \quad a(x) = \tilde \rho(x) a_0(|\hat x|) = \tilde \rho(x) |\hat x|^\alpha,  \quad  \Phi(x,t) = \varphi(t) + a(x) \psi(t) = t^p + a(x) t^q.
\end{equation}
where $\tilde\rho$ is a nonnegative function which will be described in Lemma~\ref{def:construct1}. 


\begin{lemma}\label{def:construct1}

\begin{enumerate}

\item Let $p_0= N/k$ and $p<N/k<q-\alpha k^{-1}$. Use one saddle point~\ref{set1}. 

\item Let $p_0>N/k$ and $p\leq p_0 \leq q-\alpha \frac{p_0-1}{N-k}$.  Take $\gamma> (p_0k-N)^{-1}$ if $p=p_0$, $\gamma <\frac{1-p_0}{p_0k-N}$ if $q=p_0+  \alpha \frac{p_0-1}{p_0k-N}$, and any $\gamma$ otherwise. Use supercritical~\ref{set2}. 


\item Let $1<p_0<N/k$ and $p\leq p_0\leq q-\alpha k^{-1}$.  Take $\gamma<(p_0k-N)^{-1}$ if $p=p_0$, $\gamma >(q-1)/(N-p_0k)$ if $q=p_0+\alpha k^{-1}$, and any $\gamma$ otherwise. Use subcritical~\ref{set3}. 


\item Let $p_0=N/k$ and $p\leq p_0 \leq q-\alpha k^{-1}$. Take $\gamma >(N-k)^{-1}$ if $p=p_0$, $\gamma <-1/k$ if $q=p_0+\alpha k^{-1}$, and any $\gamma>0$ otherwise. Use right limiting critical ~\ref{set4}. 

\item Let $p_0=N/k$ and $p\leq p_0 \leq q-\alpha k^{-1}$. Take $\gamma <(k-N)^{-1}$ if $p=p_0$, $\gamma >(q-1)(N-k)^{-1}$ if $q=p_0+\alpha k^{-1}$, and any $\gamma>0$ otherwise. Use left limiting critical~\ref{set5}. 


\end{enumerate}

Then for $\Phi$ given by \eqref{adef}, the pair of forms $u$ and $A$ is a $(\Phi,k)$-separating pair.
\end{lemma}

\begin{proof}

We use Lemma~\ref{lemma:integr} with 
$$
F_1(s,\tau) = \varphi(\tau) = \tau^p, \quad F_2(s,\tau) = a_0(s) \psi(\tau) =  s^\alpha \tau^q
$$ 
and the estimates provided by Lemma~\ref{L:me}. Clearly, 
$$
F_2^* (s,\tau) = a_0(s) \psi^*\left( \frac{\tau}{a_0(s)} \right) =  c_q s^\alpha (\tau s^{-\alpha})^{q'}. 
$$

We treat the five cases according to Definition~\ref{def:construct1}.

\medskip 

\noindent\textbf{(a) Case $p_0=N/k$, $p<N/k<q-\alpha k^{-1}$.} We estimate 
$$
\int\limits_\Omega \varphi(|du|)\, dV \lesssim \int\limits_0^{\sqrt{N}} t^{-pk + N-1}\, dt<\infty
$$
provided that $p< N/k$. Also,
$$
\int\limits_\Omega a_0(|\hat x|)\psi^*(|dA|/a_0(t))\, dV \lesssim \int\limits_0^{\sqrt{N}} t^{q'(k-N-\alpha)+\alpha + N-1}\, dt <\infty
$$
provided that $q> \frac{N+\alpha}{k}$.

\medskip

\noindent\textbf{(b) Case $p_0>N/k$.} In this case 
$$
p_0 = \frac{N-\frD}{k-\frD},\quad \frD =\frac{p_0k-N}{p_0-1},\quad k-\frD = \frac{N-k}{p_0-1}.$$ 

We use case (\textbf{ii}) of Lemma~\ref{lemma:integr}.  For the first integral in \eqref{c:super}, we get
\begin{align*}
\int\limits_0^{\sqrt{N}} \varphi \left(\frac{\sup_{\bar x}\mu (B^{\bar x}_t )}{t^k} \right) |\frC_t|_k\,  t^{N-k-1}  \, dt  
&\lesssim \int\limits_0^{\sqrt{N}} t^{p (\frD-k)} (\ln t^{-1})^{-p \gamma \frD } t^{k-\frD} (\ln t^{-1})^{\gamma\frD} t^{N-k-1}\, dt\\
&=c \int\limits_0^{\sqrt{N}} t^{(p_0-p)(k-\frD)}  (\ln t^{-1})^{\gamma \frD(1-p)} \, \frac{dt}t{}<\infty.
\end{align*}

For the second integral in \eqref{c:super} we have
\begin{align*} 
\int\limits_0^{\sqrt{N}} a_0(t) \psi^* \left(\frac{t^{k-N}}{a_0(t)} \right)  |\frC_t|_{k}t^{N-k-1}\, dt  &\lesssim
\int\limits_0^{\sqrt{N}} t^{q'(k-N-\alpha)+\alpha}  t^{k-\frD} (\ln t^{-1})^{\gamma\frD } t^{N-k-1}\, dt\\
&= c \int\limits_0^{\sqrt{N}} t^{q'(k-N-\alpha)+N-\frD+\alpha} (\ln t^{-1})^{\gamma\frD} \, \frac{dt}{t}<\infty.
\end{align*}
Here one notes that 
$$
q'(k-N-\alpha)+N-\frD+\alpha >0 \Leftrightarrow q'< \frac{N-\frD+\alpha}{N-k+\alpha} \Leftrightarrow q > \frac{N-\frD+\alpha}{k-\frD} = p_0 + \frac{\alpha}{k-\frD}.
$$

Then by Lemma~\ref{lemma:integr} (\textbf{ii}) the pair of forms $(u,A)$ is $(\Phi,k)$-separating.

\medskip

\noindent\textbf{(c) Case $p_0<N/k$.} In this case $\frD = N-p_0k$. We use case (\textbf{i}) of Lemma~\ref{lemma:integr}. For the first integral in \eqref{c:sub}, we have
\begin{align*}
\int\limits_0^{\sqrt{N}} \varphi(t^{-k})  |\frC_t|_{N-k} t^{k-1} \, dt &\lesssim \int\limits_0^{\sqrt{N}} t^{-pk}   t^{N-k-\frD} (\ln t^{-1})^{\gamma\frD} t^{k-1}\, dt\\
&= c \int\limits_0^{\sqrt{N}} t^{(p_0-p)k} (\ln t^{-1})^{\gamma\frD} \, \frac{dt}{t}<\infty.
\end{align*}

For the second integral in \eqref{c:sub} we get
\begin{align*}
\int\limits_0^{\sqrt{N}} a_0(t) \psi^* \left(\frac{\sup_{\bar x} \mu (B^{\bar x}_t)}{t^{N-k} a_0(t)} \right)  |\frC_t|_{N-k}t^{k-1}\, dt 
&\lesssim \int\limits_0^{\sqrt{N}} t^{q' (\frD+k-N-\alpha)+\alpha} (\ln t^{-1})^{ -q' \gamma \frD } t^{N-k-\frD} (\ln t^{-1})^{\gamma\frD} t^{k-1} \, dt\\
&=c \int\limits_0^{\sqrt{N}} t^{q'(\frD+k-N-\alpha)+N-\frD + \alpha}  (\ln t^{-1})^{ \gamma \frD /(1-q) } \, \frac{dt}{t}<\infty.
\end{align*}
Here one notes that 
$$
q'(\frD+k-N-\alpha)+N-\frD + \alpha > 0 \Leftrightarrow q'< \frac{N-\frD + \alpha}{N+\alpha - \frD - k} \Leftrightarrow q>\frac{N+\alpha-\frD}{k}=p_0+\frac{\alpha}{k}.
$$

By Lemma~\ref{lemma:integr} i), the pair $(u,A)$  is $(\Phi,k)$-separating.


\medskip

\noindent\textbf{(d) Case $p_0=N/k+0$.} We use case (\textbf{ii}) of Lemma~\ref{lemma:integr}. For the first integral in \eqref{c:super}, we get 
\begin{align*}
\int\limits_0^{\sqrt{N}} \varphi \left(\frac{\sup_{\bar x}\mu (B^{\bar x}_t )}{t^k} \right) |\frC_t|_k\,  t^{N-k-1}  \, dt&\lesssim \int\limits_0^{\sqrt{N}} t^{-pk} (\ln t^{-1})^{-p \gamma k}  t^{k} (\ln t^{-1})^{\gamma k} t^{N-k-1} \, dt\\
&=c \int\limits_0^{\sqrt{N}}  t^{(p_0-p)k}(\ln t^{-1})^{\gamma k (1-p)} \, \frac{dt}{t}<\infty.
\end{align*}

For the second integral in \eqref{c:super} we have
\begin{align*}
\int\limits_0^{\sqrt{N}} a_0(t) \psi^* \left(\frac{t^{k-N}}{a_0(t)} \right)  |\frC_t|_{k}t^{N-k-1}\, dt&\lesssim \int\limits_0^{\sqrt{N}} t^{q'(k-N-\alpha)+\alpha} t^{k} (\ln t^{-1})^{\gamma k}t^{N-k-1}\, dt\\
&= c \int\limits_0^{\sqrt{N}} t^{q'(k-N-\alpha)+N+\alpha} (\ln t^{-1})^{\gamma  k} \, \frac{dt}{t}<\infty.
\end{align*}
Here one notes that 
$$
q'(k-N-\alpha)+N+\alpha>0 \Leftrightarrow q'<\frac{N+\alpha}{N+\alpha-k} \Leftrightarrow q> \frac{N+\alpha}{k} = p_0+\frac{\alpha}{k}.
$$
By Lemma~\ref{lemma:integr} (\textbf{ii}), the pair $(u,A)$ is $(\Phi,k)$-separating.


\medskip

\noindent\textbf{(e) Case $p_0=N/k-0$.} We use case (\textbf{i}) of Lemma~\ref{lemma:integr} For the first integral in \eqref{c:sub}, we have
\begin{align*}
\int\limits_0^{\sqrt{N}} \varphi(t^{-k})  |\frC_t|_{N-k} t^{k-1} \, dt&\lesssim \int\limits_0^{\sqrt{N}} t^{-pk}   t^{N-k} (\ln t^{-1})^{\gamma  (N-k)} t^{k-1}\, dt\\
&= c \int\limits_0^{\sqrt{N}} t^{(p_0-p)k} (\ln t^{-1})^{\gamma  (N-k)} \, \frac{dt}{t}<\infty.
\end{align*}

For the second integral in \eqref{c:super}, we get
\begin{align*}
\int\limits_0^{\sqrt{N}} a_0(t) \psi^* \left(\frac{\sup_{\bar x} \mu (B^{\bar x}_t)}{t^{N-k} a_0(t)} \right)  |\frC_t|_{N-k}t^{k-1}\, dt&\lesssim \int\limits_0^{\sqrt{N}} t^{(k-N-\alpha)q'+\alpha } (\ln t^{-1})^{-q' \gamma }    t^{N-k} (\ln t^{-1})^{\gamma  (N-k)}t^{k-1}\, dt\\
&=c \int\limits_0^{\sqrt{N}} t^{q'(k-N-\alpha)+N+\alpha}  (\ln t^{-1})^{ \gamma  (N-k)/ (1-q) } \, \frac{dt}{t}<\infty.
\end{align*}
By Lemma~\ref{lemma:integr} (\textbf{i}), the pair of forms $(u,A)$ is $(\Phi,k)$-separating.


\end{proof}

\begin{theorem}\label{theorem:double} 
Let $p<N/k$ and $q>p+\alpha k^{-1}$. Then there exists $p_0\in (1,N/k)$ such that $p<p_0<q-\alpha k^{-1}$ (one can also take $p=p_0$ and choose $\gamma<(p_0k-N)^{-1}$) and therefore a $(\Phi,k)$-separating pair of forms $(u,A)$  for $\Phi$ defined by  \eqref{adef}, and $\tilde \rho$ from Lemma~\ref{def:construct1}. 

Let $q>p+\alpha (p-1)/(N-k)$ and $q> \frac{N+\alpha}{k}$. Then there exists $p_0>N/k$ satisfying $p<p_0<q-\alpha (p_0-1)/(N-k)$ (one can also take $p=p_0$ and choose $\gamma>(p_0k-N)^{-1}$) and therefore a $(\Phi,k)$-separating pair of forms $(u,A)$ for $\Phi$ defined by \eqref{adef}, and $\tilde \rho$ from Lemma~\ref{def:construct1}. 

In these cases $H^{d,\Phi(\cdot)}(\Omega,\Lambda^{k-1}) \neq W^{d,\Phi(\cdot)}(\Omega,\Lambda^{k-1})$. Let $\eta \in C_0^\infty(\Omega)$ be such that $\eta=1$ in a neighbourhood of $\frS = \frS (u,A)$, $A^\circ = \eta A$, and $b=dA^\circ$. Then for the functional $\mathcal{F}_{\Phi,b}$ there holds
$$
 \inf  \mathcal{F}_{\Phi,b}(W^{d,\Phi(\cdot)}_T(\Omega,\Lambda^{k-1})) < \inf  \mathcal{F}_{\Phi,b}(H^{d,\Phi(\cdot)}_T (\Omega,\Lambda^{k-1})).
$$
For sufficiently large $t>0$ and 
  \begin{align*}
    w_t &= \argmin
          \mathcal{F}_{\Phi,0}\big( tu^\partial+ W_0^{1,\Phi(\cdot)}(\Omega,\Lambda^{k-1})\big) \qquad
    \\
    h_t &= \argmin \mathcal{F}_{\Phi,0}\big( tu^\partial+ C_0^{\infty}(\Omega,\Lambda^{k-1})\big)
  \end{align*}
  we have $w_t \neq h_t$ and $\mathcal{F}_{\Phi,0}(w_t) < \mathcal{F}_{\Phi,0}(h_t)$.
\end{theorem}

\begin{proof}
We have only to check Assumption~\ref{Ass:basic1}. Indeed, since $\tilde\rho=0$ on the support of $du$ and $\tilde \rho=1$ on the support of $b=dA$, 
$$
\mathcal{F}_{\Phi,0}(tu)=t^p\mathcal{F}_{\Phi,0}(u), \quad \mathcal{F}_{\Phi,0}^*(sb) \leq s^{q'} \mathcal{F}_{\Phi,0}^*(b).
$$
Take $s=t^{p/q'}$ Then for sufficiently large $t$ there holds
$$
\mathcal{F}_{\Phi,0}(tu) + \mathcal{F}^*_{\Phi,0}(sb) \leq t^p(\mathcal{F}_{\Phi,0}(u)+\mathcal{F}^*_{\Phi,0}(b)) <ts= t^{1+\frac{p}{q'}}
$$
since $p< 1+ \frac{p}{q'}$ if $p<q$.
\end{proof}

Note that here $\Phi(x,t) =  t^p +  a(x)t^q$ where $a\in C^\alpha(\overline{\Omega})$ (by Lemma~\ref{L:cont0}). This proves Theorem \ref{theoremA}.

\subsection{Borderline Double Phase}
\label{subsec:bor} 

Let $p_0>1$, $\alpha,\beta \in \mathbb{R}$, $\varkappa\geq 0$ such that 
\begin{equation}\label{ab_cond}
\alpha+\beta>p_0+\varkappa.
\end{equation}
 Let $\varphi$ and $\psi$ be two Orlicz functions such that
\begin{equation}\label{OrliczLog}
\varphi \lesssim \psi,\quad \varphi (t) \lesssim t^{p_0} \ln^{-\beta} (e+t), \quad \psi^*(t) \lesssim t^{p_0'} \ln ^{\alpha/(1-p_0)}(e+t)
\end{equation}
for large $t$. Denote 
\begin{equation}\label{adef1}
a_0(t) = \ln^{-\varkappa}(1/t), \quad a(x) = \tilde \rho(x) a_0(|\hat x|), \quad \Phi(x,t) = \varphi(t) + a(x) \psi(t),
\end{equation}
where $\tilde\rho$ is a nonnegative function to be defined later.


\begin{lemma}\label{def:construct}

\begin{enumerate}

\item Let $p_0= N/k$ and assume that $\beta>1$ and $\alpha+1>\kappa +p_0$. Use one saddle point~\ref{set1}.  


\item If $p_0>N/k$, define $\frD$ from $p_0 = (N-\frD) / (k-\frD)$ and take $\gamma$ satisfying $(1-\beta)/(p_0-1) <\gamma \frD < (\alpha-\varkappa-p_0+1)/(p_0-1)$. Use supercritical~\ref{set2}. 



\item If $1<p_0<N/k$, define $\frD$ from $p_0= (N-\frD)/k$ and take $\gamma$ satisfying  $p+\varkappa-\alpha-1<\gamma \frD < \beta-1$. Use subcritical~\ref{set3}. 



\item If $p_0=N/k$ and additionally $\alpha>p_0-1+\varkappa$, take $\gamma$ satisfying $(1-\beta)/(p_0-1)<\gamma  k<(\alpha-\varkappa-p_0+1)/(p-1)$. Use right limiting critical~\ref{set4}. 
\item If $p_0=N/k$ and additionally $\beta >1$, take $\gamma$ satisfying $p_0+\varkappa-\alpha-1<\gamma (N-k)<\beta-1$. Use left limiting critical~\ref{set5}. 



\end{enumerate}
Then for $\Phi$ given by \eqref{adef1}, the pair of forms $u$ and $A$ is a $(\Phi,k)$-separating pair.

\end{lemma}

\begin{proof}

To shorten notation we write here $p$ instead of $p_0$. We use Lemma~\ref{lemma:integr} with 
$$
F_1(s,\tau) = \tau^p \ln^{-\beta}(e+\tau), \quad F_2(s,\tau) = a_0(s) \psi(\tau)
$$
and the estimates provided by Lemma~\ref{L:me}. We have 
$$
F_2^* (s,\tau) = a_0(s) \psi^* \left(\frac{\tau}{a_0(s)} \right).
$$

 We treat the five cases according to Definition~\ref{def:construct}.
\smallskip

\noindent\textbf{(a) Case $p=N/k$.} We estimate 
\begin{align*}
\int\limits_\Omega \Phi(x,|du|)\, dV =\int\limits_\Omega \varphi(|du|)\, dV  \lesssim \int\limits_0^{\sqrt{N}} t^{-pk + N-1} \ln^{-\beta} (e+t^{-1})\, dt = \int\limits_0^{\sqrt{N}} \ln^{-\beta} (e+t^{-1})\, \frac{dt}{t} <\infty,\\
\int\limits_\Omega \Phi^*(x,|dA|)\, dV \leq \int\limits_\Omega a_0(|\hat x|)\psi^*(|dA|/a_0(t))\, dV \lesssim \int\limits_0^{\sqrt{N}} \ln^{(\kappa-\alpha)/(p-1)}(e+t^{-1})  \frac{dt}{t} <\infty.
\end{align*}
Therefore the pair of forms $(u,A)$ is $(\Phi,k)$-separating.


\noindent\textbf{(b) Case $p>N/k$.} We use case (\textbf{ii}) of Lemma~\ref{lemma:integr}. For the first integral in \eqref{c:super}, we get
\begin{align*}
\int\limits_0^{\sqrt{N}} \varphi \left(\frac{\sup_{\bar x}\mu (B^{\bar x}_t )}{t^k} \right) |\frC_t|_k\,  t^{N-k-1}  \, dt<\infty &\lesssim \int\limits_0^{\sqrt{N}} t^{p (\frD-k)} (\ln t^{-1})^{-p \gamma \frD -\beta} t^{k-\frD} (\ln t^{-1})^{\gamma\frD} t^{N-k-1}\, dt\\
&=c \int\limits_0^{\sqrt{N}} t^{-1}  (\ln t^{-1})^{\gamma \frD(1-p)-\beta} \, dt<\infty.
\end{align*}

For the second integral in \eqref{c:super}, using $p'=(N-\frD)/(N-k)$ we have 
\begin{align*}
&\int\limits_0^{\sqrt{N}} a_0(t) \psi^* \left(\frac{t^{k-N}}{a_0(t)} \right)  |\frC_t|_{k}t^{N-k-1}\, dt \\
&\lesssim
\int\limits_0^{\sqrt{N}} t^{p'(k-N)} (\ln t^{-1})^{ \varkappa p'-\alpha/(p-1)}   t^{k-\frD} (\ln t^{-1})^{\gamma\frD - \varkappa} t^{N-k-1}\, dt\\
&= c \int\limits_0^{\sqrt{N}} t^{-1} (\ln t^{-1})^{\gamma\frD-(\alpha-\varkappa)/(p-1)} \, dt<\infty.
\end{align*}
 Then by Lemma~\ref{lemma:integr} (\textbf{ii}) the pair of forms $(u,A)$  is $(\Phi,k)$-separating.




\noindent\textbf{(c) Case $p<N/k$.} We use case (\textbf{i}) of Lemma~\ref{lemma:integr} for the first integral in \eqref{c:sub}, we have
\begin{align*}
\int\limits_0^{\sqrt{N}} \varphi(t^{-k})  |\frC_t|_{N-k} t^{k-1} \, dt &\lesssim \int\limits_0^{\sqrt{N}} t^{-pk} (\ln t^{-1})^{-\beta}   t^{N-k-\frD} (\ln t^{-1})^{\gamma\frD} t^{k-1}\, dt\\
&= c \int\limits_0^{\sqrt{N}} t^{-1} (\ln t^{-1})^{\gamma\frD-\beta} \, dt<\infty.
\end{align*}

For the second integral in \eqref{c:sub}, using $p' =(N-\frD)/(N-k-\frD)$ we get
\begin{align*}
&\int\limits_0^{\sqrt{N}} a_0(t) \psi^* \left(\frac{\sup_{\bar x} \mu (B^{\bar x}_t)}{t^{N-k} a_0(t)} \right)  |\frC_t|_{N-k}t^{k-1}\, dt \\&\lesssim \int\limits_0^{\sqrt{N}} t^{p' (\frD+k-N)} (\ln t^{-1})^{p'\varkappa  -p' \gamma \frD +\alpha / (1-p)} t^{N-k-\frD} (\ln t^{-1})^{\gamma\frD-\varkappa} t^{k-1} \, dt\\
&=c \int\limits_0^{\sqrt{N}} t^{-1}  (\ln t^{-1})^{ (\gamma \frD+\alpha-\varkappa) / (1-p) } \, dt<\infty.
\end{align*}
By Lemma~\ref{lemma:integr} (\textbf{i}), the pair $(u,A)$ is $(\Phi,k)$-separating. 




\medskip


\noindent\textbf{(d) Case $p=N/k$, $\alpha>p-1$.} We use case (\textbf{ii}) of Lemma~\ref{lemma:integr}. For the first integral in \eqref{c:super}, we get 
\begin{align*}
\int\limits_0^{\sqrt{N}} \varphi \left(\frac{\sup_{\bar x}\mu (B^{\bar x}_t )}{t^k} \right) |\frC_t|_k\,  t^{N-k-1}  \, dt&\lesssim \int\limits_0^{\sqrt{N}} t^{-pk} (\ln t^{-1})^{-p \gamma  k -\beta}  t^{k} (\ln t^{-1})^{\gamma  k} t^{N-k-1} \, dt\\
&=c \int\limits_0^{\sqrt{N}} t^{-1}  (\ln t^{-1})^{\gamma  k (1-p)-\beta} \, dt<\infty.
\end{align*}

For the second integral in \eqref{c:super}, using $p'= N/(N-k)$ we have
\begin{align*}
\int\limits_0^{\sqrt{N}} a_0(t) \psi^* \left(\frac{t^{k-N}}{a_0(t)} \right)  |\frC_t|_{k}t^{N-k-1}\, dt&\lesssim \int\limits_0^{\sqrt{N}} t^{p'(k-N)} (\ln t^{-1})^{p'\varkappa-\alpha/(p-1)}   t^{k} (\ln t^{-1})^{\gamma  k-\varkappa}t^{N-k-1}\, dt\\
&= c \int\limits_0^{\sqrt{N}} t^{-1} (\ln t^{-1})^{\gamma  k-(\alpha-\varkappa)/(p-1)} \, dt<\infty.
\end{align*}
By Lemma~\ref{lemma:integr} (\textbf{ii}), the pair $(u,A)$  is $(\Phi,k)$-separating. 



\noindent\textbf{(e) Case $p=N/k$, $\beta>1$.} We use case (\textbf{i}) of Lemma~\ref{lemma:integr} For the first integral in \eqref{c:sub}, we have
\begin{align*}
\int\limits_0^{\sqrt{N}} \varphi(t^{-k})  |\frC_t|_{N-k} t^{k-1} \, dt&\lesssim \int\limits_0^{\sqrt{N}} t^{-pk} (\ln t^{-1})^{-\beta}   t^{N-k} (\ln t^{-1})^{\gamma  (N-k)} t^{k-1}\, dt\\
&= c \int\limits_0^{\sqrt{N}} t^{-1} (\ln t^{-1})^{\gamma  (N-k)-\beta} \, dt<\infty.
\end{align*}

For the second integral in \eqref{c:super}, we get
\begin{align*}
&\int\limits_0^{\sqrt{N}} a_0(t) \psi^* \left(\frac{\sup_{\bar x} \mu (B^{\bar x}_t)}{t^{N-k} a_0(t)} \right)  |\frC_t|_{N-k}t^{k-1}\, dt\\&\lesssim \int\limits_0^{\sqrt{N}} t^{-(N-k)p' } (\ln t^{-1})^{p'\varkappa-p' \gamma  (N-k)  +\alpha / (1-p)}    t^{N-k} (\ln t^{-1})^{\gamma  (N-k)-\varkappa}t^{k-1}\, dt\\
&=c \int\limits_0^{\sqrt{N}} t^{-1}  (\ln t^{-1})^{ (\gamma (N-k)+\alpha-\varkappa) / (1-p) } \, dt<\infty.
\end{align*}
By Lemma~\ref{lemma:integr} (\textbf{i}), the pair of forms $(u,A)$ is $(\Phi,k)$-separating.
\end{proof}

\begin{theorem}\label{theorem:bor}

Under condition \eqref{ab_cond}, for any $k=1,\ldots,N-1$ and any $p>1$ there exists $\tilde \rho$ and a $(\Phi,k)$-separating pair of forms $(u,A)$ for $\Phi$ defined by \eqref{OrliczLog} and \eqref{adef1}. Therefore in these cases
$$
H^{d,\Phi(\cdot)}(\Omega,\Lambda^{k-1}) \neq W^{d,\Phi(\cdot)}(\Omega,\Lambda^{k-1}).
$$
Let $\eta \in C_0^\infty(\Omega)$ be such that $\eta=1$ in a neighbourhood of $\frS = \frS (u,A)$, $A^\circ = \eta A$, and $b=dA^\circ$. For the functional $\mathcal{F}_{\Phi,b}$ there holds
$$
 \inf  \mathcal{F}_{\Phi,b}(W^{d,\Phi(\cdot)}_T(\Omega,\Lambda^{k-1})) < \inf  \mathcal{F}_{\Phi,b}(H^{d,\Phi(\cdot)}_T (\Omega,\Lambda^{k-1})).
$$
\end{theorem}

Note that here we have $\Phi$ given by $\Phi(x,t) = \varphi(t) + a(x)\psi(t)$, with $a \in C^{\omega(\cdot)}(\overline{\Omega})$, $\omega (t)\leq C \ln^{-\kappa}(1/t)$ for some $C>1$ (see Lemma~\ref{L:cont0}).  This proves Theorem \ref{theoremB}.

\subsection{Variable exponent model}
\label{subsec:var}

A classical example of an integrand from the class \eqref{eq:growth}  is the variable exponent model
\begin{equation}\label{PhiVar}
  \Phi(x,t) =  t^{p(x)},
\end{equation}
where $p\,:\, \Omega \to [p_{-},p_{+}]$ is a variable exponent. Let $p_0 \in (p_{-},p_{+})$, 
\begin{align}\label{eq:logcon}
\sigma(t) = \kappa \frac {\ln  \ln \frac 1 t }{\ln \frac 1 t},
\end{align}
 with~$\kappa>0$ and $\tilde \rho$ be a function to be defined later. Let $\xi \in C^\infty(\mathbb{R})$ be a positive nondecreasing function such that $\xi(t)=t$ if $t\in [(p_{-}+p_0)/2,(p_{+}+p_0)/2]$, $\xi(t)=\xi(p_{-})= (3p_{-}+p_0)/4$ if $t\leq p_{-}$, $\xi(t)=\xi(p_{+}) = (3p_{+} + p_0)/4$ if $t \geq p_+$. Set  
\begin{equation}\label{Var_p_def}
p(x) = \xi\bigl(p_0 + \sigma(|\hat x|) (2\tilde \rho-1)\bigr),  
\end{equation}
and let $\Phi$ be defined by \eqref{PhiVar}.

Recall that due to the well-know result from \cite{Zhi04} if the exponent $p$ has the modulus of continuity \eqref{eq:logcon} with sufficiently small $\kappa$ then smooth functions are dense in corresponding Sobolev-Orlicz space and the Lavrentiev phenomenon is absent. On the other hand, the example with one saddle point provided in \cite{Zhi04} ($k=0$, $N=2$, $p_{-}<2<p_{+}$) shows that for sufficiently large $\kappa$ the Lavrentiev gap occurs. We construct examples of the Lavrentiev phenomenon for $p(x)$-integrand in arbitrary dimension and for any $1<p_{-}<p_{+}<\infty$.

\begin{lemma}\label{def:construct3}

\begin{enumerate}

\item Let $p_0=N/k$ and $\kappa >k^{-2}\max(k,N-k)$. Use one saddle point~\ref{set1}.

\item Let $p_0>N/k$,
\begin{equation}\label{var_b}
\kappa > \frac{p_0(p_0-1)}{2(N-k)}, \quad\text{and}\quad 1- \kappa \frac{N-k}{p_0-1}< \gamma (kp_0-N) <\kappa \frac{N-k}{p_0-1}-(p_0-1). 
\end{equation}
Use supercritical~\ref{set2}.

\item Let $1<p_0<N/k$,
\begin{equation}\label{var_c}
\kappa > \frac{p_0}{2k} \quad\text{and}\quad p_0-1-\kappa k<\gamma (N-p_0k) < \kappa k -1.
\end{equation}
Use subcritical~\ref{set3}.

\item Let $p_0=N/k$,
\begin{equation}\label{var_d}
\kappa > \frac{N}{2k^2}, \quad\text{and}\quad \frac{k-\kappa k^2}{N-k} < \gamma k < \frac{k-N+\kappa k^2}{N-k}.
\end{equation}
Use right limiting critical~\ref{set4}.


\item Let $p_0=N/k$, 
\begin{equation}\label{var_e}
\kappa > \frac{N}{2k^2}, \quad\text{and}\quad -\kappa k + \frac{N-k}{k}<\gamma k < \kappa k -1.
\end{equation}
Use left limiting critical~\ref{set5}.


\end{enumerate}
Then for $\Phi$ given by \eqref{PhiVar} and \eqref{Var_p_def}, the pair of forms $u$ and $A$ is a $(\Phi,k)$-separating pair.
\end{lemma}

\begin{proof}

We use Lemma~\ref{lemma:integr} with $F_1(s,\tau) = \tau^{p_0-\sigma(s)}$ and $F_2(s,\tau) =\tau^{p_0+\sigma(s)}$. Clearly, $F_2^* (s,\tau) \le c(p_{-},p_{+}) \tau^{(p_0+\sigma(s))'}$. Note that $t^{\sigma(t)} = (\ln t^{-1})^{-\kappa} $.

\smallskip

\noindent\textbf{(a) Case $p_0=N/k$.} We evaluate
$$
\int\limits_\Omega \Phi(x,|du|)\, dV \lesssim \int\limits_\Omega |du|^{p_0-\sigma(|\hat x|)}\, dV \lesssim \int\limits_0^{\sqrt{N}} t^{N-1-k (p_0-\sigma(t))}\, dt 
\lesssim  \int\limits_0^{\sqrt{N}} (\ln t^{-1})^{-k\kappa} t^{-1}\, dt<\infty
$$
provided that $k\kappa >1$. Also
\begin{gather*}
\int \limits_\Omega \Phi^*(x,|dA|)\,dV \lesssim \int\limits_\Omega |dA|^{(p_0+\sigma(\hat x))'}\,dV\\
\lesssim \int\limits_0^{\sqrt{N}} t^{(k-N)(p_0+\sigma(t))/(p_0+\sigma(t)-1)} t^{N-1}\, dt 
= \int\limits_0^{\sqrt{N}} (\ln t^{-1})^{r(t)} t^{-1}\, dt,
\end{gather*}
where $r(t) = -\kappa k^2 /(N-k+k \sigma(t))$. Since $\lim\limits_{t\to +0} r(t)<-1$, the last integral converges.

\noindent\textbf{(b) Case $p_0>N/k$.} We have $\frD = \frac{p_0 k - N}{p_0-1}$, $p_0=\frac{N-\frD}{k-\frD}$, and the conditions \eqref{var_b} on $\kappa$ and $\gamma$ can be rewritten as 
\begin{equation}\label{var_b1}
\kappa > \frac{N-\frD}{2(k-\frD)^2} \quad \text{and} \quad \frac{k-\frD}{N-k} - \kappa \frac{(k-\frD)^2}{N-k} < \gamma \frD < \kappa \frac{(k-\frD)^2}{N-k}-1.
\end{equation}

We use case (\textbf{ii}) of Lemma~\ref{lemma:integr}.  For the first integral in \eqref{c:super}, we get
\begin{align*}
&\int\limits_0^{\sqrt{N}} F_1 \bigl(t, t^{-k}\sup_{\bar x}\mu (B^{\bar x}_t ) \bigr) |\frC_t|_k\,  t^{N-k-1}dt =\int\limits_0^{\sqrt{N}} (t^{-k} \sup_{\bar x}\mu (B^{\bar x}_t))^{p_0-\sigma(t)} |\frC_t|_k\,  t^{N-k-1}  \, dt\\ 
& \lesssim \int\limits_0^{\sqrt{N}} (t^{-k} t^{\frD} (\ln (t^{-1})^{-\gamma \frD}) )^{\frac{N-\frD}{k-\frD}-\sigma(t)} t^{k-\frD}(\ln(t^{-1}))^{\gamma\frD} t^{N-k}  \frac {dt}{t} =\int\limits_0^{\sqrt{N}} t^{(k-\frD)\sigma(t)}(\ln(t^{-1}))^{(\frac{k-N}{k-\frD}+\sigma(t))\gamma\frD}  \frac {dt}{t}\\
 &  = \int\limits_0^{\sqrt{N}}(\ln(t^{-1}))^{r(t)}  \frac {dt}{t}, \quad \quad r(t) = \left(\frac{k-N}{k-\frD}+\sigma(t)\right)\gamma\frD - \kappa(k-\frD).
\end{align*}
Since \eqref{var_b1} implies that $\lim\limits_{t\to +0} r(t)<-1$, the last integral converges.



 For the second integral in \eqref{c:super}, we get
\begin{align*}
\int\limits_0^{\sqrt{N}} F_2^*  \bigl( t,t^{k-N} \bigr)  |\frC_t|_{k}t^{N-k-1}\, dt&=\int\limits_0^{\sqrt{N}}(t^{k-N})^{(\frac{N-\frD}{k-\frD}+\sigma(t))'}   t^{k-\frD}(\ln(t^{-1}))^{\gamma\frD} t^{N-k-1}\,  dt \\
 & =\int\limits_0^{\sqrt{N}} (\ln t^{-1})^{r(t)} \frac{dt}{t},\quad r(t)=\gamma \frD - \kappa \frac{(k-\frD)^2}{N-k+(k-\frD)\sigma(t)}.
\end{align*}
Since \eqref{var_b1} implies that $\lim\limits_{t\to +0} r(t)<-1$, the last integral converges. 

By Lemma~\ref{lemma:integr} (\textbf{ii}) the pair $(u,A)$ is $(\Phi,k)$-separating. 

\noindent\textbf{(c) Case $p_0<N/k$.} We have~$\frD = N-p_0k$ and the conditions \eqref{var_c} on $\kappa$ and $\gamma$ can be rewritten as
\begin{equation}\label{var_c1}
\kappa > \frac{N-\frD}{2k^2}\quad \text{and} \quad \frac{N-\frD-k}{k}-\kappa k<\gamma \frD < \kappa k -1.
\end{equation} 

 We use case (\textbf{i}) of Lemma~\ref{lemma:integr} for the first integral in \eqref{c:sub}, we have

$$
\int\limits_0^{\sqrt{N}} F_1(t,t^{-k})  |\frC_t|_{N-k} t^{k-1} \, dt=\int\limits_0^{\sqrt{N}} (t^{-k})^{p_0-\sigma(t)}  t^{N-k-\frD} (\ln t^{-1})^{\gamma\frD} t^{k-1}\, dt
=\int\limits_0^{\sqrt{N}} (\ln t^{-1})^{\gamma \frD-\kappa k} \frac{dt}{t}<\infty 
$$
since \eqref{var_c1} implies $\gamma \frD - \kappa k < -1$.

For the second integral in \eqref{c:sub}, using $p' =(N-\frD)/(N-k-\frD)$ we get
\begin{align*}
&\int\limits_0^{\sqrt{N}} F_2^* \bigl(t, t^{k-N}\sup_{\bar x} \mu (B^{\bar x}_t) \bigr)  |\frC_t|_{N-k}t^{k-1}\, dt=\int\limits_0^{\sqrt{N}} (t^{k-N}\sup_{\bar x} \mu (B^{\bar x}_t))^{(p_0+\sigma(t))'}|\frC_t|_{N-k}t^{k-1}\, dt\\
&\le \int\limits_0^{\sqrt{N}} (t^{k-N} t^\frD (\ln t^{-1})^{-\gamma \frD}))^{(p_0+\sigma(t))'}t^{N-k-\frD} (\ln t^{-1})^{\gamma\frD} t^{k-1}\, dt = \int\limits_0^{\sqrt{N}} (\ln t^{-1})^{r(t)} \frac{dt}{t},\\
&\text{where}\quad r(t)=\frac{-\kappa k^2 -k \gamma \frD}{N-\frD-k+k\sigma}.
\end{align*}
Since \eqref{var_c1} implies that $\lim\limits_{t\to +0}r(t)<-1$, the last integral converges.

By Lemma~\ref{lemma:integr} (\textbf{i}), the pair $(u,A)$ is a $(\Phi,k)$-separating.

\medskip

\noindent\textbf{(d) Case $p_0=N/k+0$.} For the first integral in \eqref{c:super}, we get
\begin{align*}
&\int\limits_0^{\sqrt{N}} F_1 \bigl(t, t^{-k}\sup_{\bar x}\mu (B^{\bar x}_t ) \bigr) |\frC_t|_k\,  t^{N-k-1}dt =\int\limits_0^{\sqrt{N}} (t^{-k} \sup_{\bar x}\mu (B^{\bar x}_t))^{p_0-\sigma(t)} |\frC_t|_k\,  t^{N-k-1}  \, dt\\ 
 &\lesssim \int\limits_0^{\sqrt{N}} (t^{-k} (\ln (t^{-1})^{-\gamma k}) )^{\frac{N}{k}-\sigma(t)} t^{k}(\ln(t^{-1}))^{\gamma k} t^{N-k}  \frac {dt}{t} 
  =\int\limits_0^{\sqrt{N}}(\ln(t^{-1}))^{\gamma(k-N+k\sigma(t))-\kappa k}  \frac {dt}{t}<\infty 
\end{align*}
since \eqref{var_d} implies $\gamma (k-N)-\kappa k < -1$.

For the second integral in \eqref{c:super} we have
\begin{align*}
\int\limits_0^{\sqrt{N}} F_2^*  \bigl( t,t^{k-N} \bigr)  |\frC_t|_{k}t^{N-k-1}\, dt&=\int\limits_0^{\sqrt{N}}(t^{k-N})^{(\frac{N}{k}+\sigma(t))'}   t^{k}(\ln(t^{-1}))^{\gamma k} t^{N-k-1}\,  dt  =\int\limits_0^{\sqrt{N}} (\ln t^{-1})^{r(t)} \frac{dt}{t}, \\
r(t)=\gamma k - \frac{\kappa k^2}{N-k+k\sigma}.
\end{align*}
Since \eqref{var_d} implies that $\lim\limits_{t\to +0}r(t)<-1$, the last integral converges.  

\medskip

\noindent\textbf{(e) Case $p_0=N/k-0$.} We use case (\textbf{i}) of Lemma~\ref{lemma:integr} for the first integral in \eqref{c:sub}, we have

$$
\int\limits_0^{\sqrt{N}} F_1(t,t^{-k})  |\frC_t|_{N-k} t^{k-1} \, dt=\int\limits_0^{\sqrt{N}} (t^{-k})^{p_0-\sigma(t)}  t^{N-k-\frD} (\ln t^{-1})^{\gamma k} t^{k-1}\, dt
=\int\limits_0^{\sqrt{N}} (\ln t^{-1})^{\gamma k-\kappa k} \frac{dt}{t}<\infty 
$$
since \eqref{var_e} implies $\gamma k - \kappa k < -1 $.

For the second integral in \eqref{c:sub}, using $p_0' =N/(N-k)$ we get
\begin{align*}
\int\limits_0^{\sqrt{N}} F_2^* \bigl(t, t^{k-N}\sup_{\bar x} \mu (B^{\bar x}_t) \bigr)  |\frC_t|_{N-k}t^{k-1}\, dt=\int\limits_0^{\sqrt{N}} (t^{k-N}\sup_{\bar x} \mu (B^{\bar x}_t))^{(p_0+\sigma(t))'}|\frC_t|_{N-k}t^{k-1}\, dt\\
\le \int\limits_0^{\sqrt{N}} (t^{k-N} (\ln t^{-1})^{-\gamma k}))^{(p_0+\sigma(t))'}t^{N-k} (\ln t^{-1})^{\gamma k} t^{k-1}\, dt= \int\limits_0^{\sqrt{N}} (\ln t^{-1})^{r(t)} \frac{dt}{t},
\quad r(t)=\frac{-\kappa k^2 - \gamma k^2}{N-k+k\sigma}.
\end{align*}
Since \eqref{var_e} implies that $\lim\limits_{t\to +0}r(t)<-1$, the last integral converges.

\end{proof}

\begin{theorem}\label{theorem:var}

Let~$\Omega=B_1$, $k\in \{1,\ldots,N-1\}$. Let $1< p^- < p^+ < \infty$. Then there exists 
 a variable exponent~$p\,:\, \Omega \to [p^-,p^+]$ (defined by \eqref{Var_p_def} and \eqref{eq:logcon}) and $(\Phi,k)$-separating pair $(u,A)$ for $\Phi(x,t)=t^{p(x)}$ defined by \eqref{PhiVar} . Moreover, $p\in  C^\infty(\overline{\Omega}\setminus \frS) \cap C(\overline{\Omega})$, where $\frS = \frS (u,A)$ is a closed set of Lebesgue measure zero. In these cases
$$
H^{d,p(\cdot)}(\Omega,\Lambda^{k-1}) \neq W^{d,p(\cdot)}(\Omega,\Lambda^{k-1}).
$$
Let $\eta \in C_0^\infty(\Omega)$ be such that $\eta=1$ in a neighbourhood of $\frS = \frS (u,A)$, $A^\circ=\eta A$, $b=dA^\circ$. for the functional $\mathcal{F}_{\Phi,b}$ there holds
$$
 \inf  \mathcal{F}_{\Phi,b}(W^{d,p(\cdot)}_T(\Omega,\Lambda^{k-1})) < \inf  \mathcal{F}_{\Phi,b}(H^{d,p(\cdot)}_T (\Omega,\Lambda^{k-1})).
$$
For sufficiently large $t>0$ and
  \begin{align*}
    w_t &= \argmin
          \mathcal{F}_{\Phi,0}\big( tu^\partial+ W_0^{1,\Phi(\cdot)}(\Omega,\Lambda^{k-1})\big) \qquad
    \\
    h_t &= \argmin \mathcal{F}_{\Phi,0}\big( tu^\partial+ C_0^{\infty}(\Omega,\Lambda^{k-1})\big)
  \end{align*}
  we have $w_t \neq h_t$ and $\mathcal{F}_{\Phi,0}(w_t) < \mathcal{F}_{\Phi,0}(h_t)$.
\end{theorem}
\begin{proof}

We have to check only the last statement (different solutions of the Dirichlet problem). By Theorem~\ref{thm:harmonic}, it remains to show that for our $(\Phi,k)$-separating pair $(u,A)$  there holds
  \begin{align*}
    \mathcal{F}_{\Phi,0}(tu) + \mathcal{F}_{\Phi,0}^*(s\, dA) \leq \tfrac 12  st
  \end{align*}
  for suitable large~$s,t$. The argument repeats that given in the proof of Theorem 32 in \cite{BalDieSur20} and we omit it.\end{proof}


In this construction by Lemma~\ref{L:cont0} the variable exponent $p(\cdot)$ has the modulus of continuity $C \bigl(\ln t^{-1}\bigr)^{-1}\ln\ln t^{-1}$.  This proves Theorem \ref{theoremC}.

\bigskip
The authors express their deep gratitude to Lars Diening for  the fruitful  discussions.

\printbibliography

\end{document}